\documentclass[12pt]{amsart}

\usepackage{fullpage}
\usepackage{mathtools}
\usepackage{amssymb,amsmath,amsthm,amscd,mathrsfs,graphicx}
\usepackage[dvipsnames]{xcolor}
\usepackage{bm}
\usepackage{hyperref}
\usepackage{dsfont}
\usepackage{enumerate}
\usepackage{epsfig}
\usepackage{float, graphicx}
\usepackage{latexsym, amsxtra}
\usepackage{pdfpages}
\usepackage{multicol}
\usepackage[normalem]{ulem}
\usepackage{psfrag}
\usepackage{stmaryrd}
\usepackage{tikz}
\usepackage[T1]{fontenc}
\usepackage{url}
\usepackage{verbatim}
\usepackage{indentfirst}
\usepackage{tikz-cd}
\usepackage{url}
\usepackage[all,pdf]{xy}
\usepackage{color}
\usepackage{verbatim}

\usepackage{cleveref}

\makeatletter
\def\thm@space@setup{%
  \thm@preskip=2ex \thm@postskip=2ex
}
\makeatother

\hypersetup{hidelinks}

\numberwithin{equation}{subsection}
\theoremstyle{plain}

\newtheorem{thm}[subsection]{Theorem~}
\newtheorem{lem}[subsection]{Lemma~}
\newtheorem{prop}[subsection]{Proposition~}

\newtheorem{cor}[subsection]{Corollary~}

\newtheorem{defn}[subsection]{Definition~}

\theoremstyle{remark}
\newtheorem{rmk}[subsection]{Remark~}

\AddToHook{env/thm/begin}{\crefalias{subsection}{thm}}
\AddToHook{env/lem/begin}{\crefalias{subsection}{lem}}
\AddToHook{env/prop/begin}{\crefalias{subsection}{prop}}
\AddToHook{env/examp/begin}{\crefalias{subsection}{examp}}
\AddToHook{env/cor/begin}{\crefalias{subsection}{cor}}

\newcommand\s[1]{\mathscr{#1}}
\renewcommand\b[1]{\mathbb{#1}}

\newcommand\m[1]{\mathrm{#1}}
\renewcommand\c[1]{\mathcal{#1}}
\renewcommand\bf[1]{\mathbf{#1}}

\usepackage{todonotes}

\title{A Comparison Theorem for $p$-adic GKZ Hypergeometric $D$-Modules and Complexes}

\author[C. Liu]{Chenhan Liu}
\address{Department of Mathematical Sciences, Tsinghua University, Beijing 100084, P. R. China}
\email{liu-ch22@mails.tsinghua.edu.cn}

\author[H. Zhang]{Hao Zhang}
\address{School of Mathematics and Shing-Tung Yau Center, Southeast University, Nanjing 211189, People’s
Republic of China}
\email{hzhang21@seu.edu.cn}
\date{February 2026}

\begin{document}

\maketitle

\begin{abstract}

We compare the $p$-adic GKZ hypergeometric complex with the hypergeometric arithmetic $\mathcal D$-module defined via six functors. Under the $p$-adic semi-nonresonance condition, the natural comparison morphism is a quasi-isomorphism. We prove concentration of the $p$-adic GKZ hypergeometric complex in degree $n$, establish overholonomicity under another non-Liouville hypothesis. When $(q-1)\gamma\in\mathbb Z^n$, we construct a Frobenius structure on its $n$-th cohomology and prove that its restriction to the nondegenerate locus is an overconvergent $F$-isocrystal of rank $n!\operatorname{vol}(\Delta)$.

\end{abstract}

\section*{Introduction}

\subsection{The GKZ hypergeometric system}

Let $A=(w_{ij})\in M_{n\times N}(\mathbb Z)$ be a matrix of rank $n$, with column vectors $\mathbf{w}_1,\cdots,\mathbf{w}_N\in \mathbb{Z}^n$. For parameters $\gamma_1, \ldots, \gamma_n \in \mathbb{C}$, Gelfand, Kapranov and Zelevinsky \cite{GKZ} define the $A$-hypergeometric system to be the system of differential equations consisting of Euler operators $Z_{i,\gamma}$ and toric operators $\Box_{\lambda}$:
\[\begin{aligned}
& Z_{i,\gamma}(f)=\sum_{j=1}^N w_{i j} x_j \frac{\partial f}{\partial x_j}+\gamma_i f=0 \quad(i=1, \ldots, n), \\
&\Box_{\lambda}(f)=\prod_{\lambda_j>0}\left(\frac{\partial}{\partial x_j}\right)^{\lambda_j} f-\prod_{\lambda_j<0}\left(\frac{\partial}{\partial x_j}\right)^{-\lambda_j} f=0,
\end{aligned}\]
where for the second system of equations, $\lambda=(\lambda_1, \ldots, \lambda_N) \in \mathbb{Z}^N$ ranges over the family of integral linear relations 
\[\Lambda=\big\{(\lambda_1,\ldots,\lambda_N):\sum_{j=1}^N \lambda_j \mathbf{w}_j=0\big\}\] among $\mathbf{w}_1, \ldots, \mathbf{w}_N$, and $f$ is a holomorphic function of $x_1,\ldots,x_N$. We now call the $A$-hypergeometric system the \textit{GKZ hypergeometric system}. 

In the following theorem, we present the fundamental properties of the GKZ hypergeometric system relevant to this article. These properties can be found in Gelfand-Kapranov-Zelevinsky \cite{GKZ}, Adolphson \cite{Ado94}, and Adolphson-Sperber \cite{AS10}.

\begin{thm}

\begin{itemize}

\item[(i)] The GKZ hypergeometric system is holonomic.

\item[(ii)] 
Suppose $(\gamma_1,\dots,\gamma_n)$ satisfies the semi-nonresonance condition. 
Then the dimension of the space of GKZ hypergeometric functions at a generic point is
$n! \cdot \operatorname{vol}(\Delta)/[\mathbb Z^n : M'],$
where $\Delta = \text{Conv}\{0, \mathbf{w}_1, \dots, \mathbf{w}_N\}$ is the convex hull, and $M'$ is the sublattice of $\mathbb Z^n$ generated by $\mathbf{w}_1, \dots, \mathbf{w}_N$.

\item[(iii)] Suppose $(\gamma_1,\dots,\gamma_n)$ satisfies the nonresonance condition. Then the GKZ hypergeometric system admits a cohomological realization in terms of an exponentially twisted relative de Rham complex on an algebraic torus.
\end{itemize}
\end{thm}

We refer the reader to \cite[pg.284-286]{Ado94} for the definitions of the nonresonance and semi-nonresonance conditions. When $\gamma$ fails to satisfy these conditions, the rank formula in (ii) need not hold.
The rank jumps for $A$-hypergeometric systems have been studied systematically by Matusevich, Miller, and Walther \cite{toricmodule}.
Motivated by Adolphson–Sperber \cite{AS10}, we study a geometric realization problem for GKZ hypergeometric systems in the $p$-adic setting. Related comparison results in the algebraic setting have also been studied by Fu \cite{Fu26}.

\subsection{$p$-adic GKZ hypergeometric complex}

Let $p$ be a prime number, $\mathbb Q_p$ the field of
$p$-adic numbers, $\mathbb Z_p$ the $p$-adic integers.
A natural $p$-adic analog of the classical GKZ system is obtained by working in Berthelot's theory of arithmetic $D$-modules on the $p$-adic formal affine space $\s A^N$. Let $D^\dagger$ denote the ring of differential operators on $\mathscr A^N$ overconvergent along infinity.
For $\gamma=(\gamma_1,\ldots,\gamma_n)\in \mathbb Z_p^n$, one may define the naive $p$-adic GKZ $D^\dagger$-module by
\[\mathcal M_\gamma^\dagger = D^\dagger\bigg/\left( \sum_{i=1}^n D^\dagger Z_{i,\gamma} + \sum_{\lambda\in \Lambda} D^\dagger \Box_\lambda \right). \]

However, this naive object is not well behaved in arithmetic geometry.
In general, it does not seem to carry a Frobenius structure even when the parameters $\gamma_i\in\frac{1}{p-1}{\mathbb Z}$. 
Instead, Fu--Li--Wan--Zhang \cite{padictorus} introduced a different object, called the $p$-adic GKZ hypergeometric complex and denoted by $C^\bullet(L^\dagger)$.
This is a logarithmic twisted relative de Rham complex over a weakly complete algebra $L^\dagger$.
Moreover, via its natural connection in the $x$-variables, it defines an object in the derived category of Berthelot arithmetic $D$-modules on $\s A^N$.
Fu--Li--Wan--Zhang \cite{padictorus} proved, under the assumption that $\gamma=0$ or the cone generated by the polytope $\Delta$ equals $\b R^n$, that $C^\bullet(L^\dagger)$ is an overholonomic object in the derived category of arithmetic $D$-modules.
They also showed that, over the non-degenerate locus, its top cohomology defines an overconvergent $F$-isocrystal of rank $n!\cdot \operatorname{vol}(\Delta)$.
The purpose of the present paper is to remove these extra assumptions and to establish the corresponding statements in greater generality.

\subsection{The $p$-adic GKZ hypergeometric $\s{D}$-module}\label{defhyperDmodule}

Another approach is to define the $p$-adic GKZ hypergeometric
$D$-module using Grothendieck's six functors.
In the algebraic setting, this perspective may be compared with the work of Matusevich, Miller, and Walther \cite{toricmodule};
it is also related to the construction of Adolphson-Sperber \cite{AS10}, where the corresponding GKZ $D$-module is realized in terms of the relative de Rham cohomology of algebraic varieties.
In the $\ell$-adic setting, an analogous construction was studied by Fu \cite{Fu16}. In the following, we formulate the corresponding construction in the $p$-adic setting and prove a comparison theorem with the $p$-adic GKZ hypergeometric complex.

Let $q=p^a$, $\mathbb F_q$ the finite field with $q$ elements, $\pi \in \overline{\mathbb{Q}}_p$ such that $\pi^{p-1}=-p$.
Let $K$ be a finite extension of $\mathbb{Q}_p$ containing $\pi$ with residue field $\kappa=\b{F}_q$, $R$ the discrete valuation ring of $K$, $\varpi$ the uniformizer of $R$.

Let $\s T^n$ and $\s A^N$ be the formal torus and affine space over $\text{spf } R$. Let $\rho_1:\s{T}^n\times \s{A}^N\to\s{T}^n$ and $\rho_2:\s{T}^n\times \s{A}^N\to\s{A}^N$ be the projections, and let $F$ be the morphism \[F:\s{T}^n\times \s{A}^N\to \s{A}^1,\quad (t_1,\cdots,t_n,x_1,\cdots,x_N)\mapsto\sum_{j=1}^Nx_jt_1^{w_{1j}}\cdots t_n^{w_{nj}}.\]
For any $\gamma=(\gamma_1,\dots,\gamma_n)\in \b{Z}_p^n$, we define the \textit{$p$-adic GKZ hypergeometric $\mathcal{D}$-module} to be the object in the category $D^b(\c{D}_{\s{P}^N,\b{Q}}^{\dagger}(\infty_N))$ given by \[\m{Hyp}_{\gamma}:=\rho_{2,+}(\rho_1^\ast\c{K}_{\gamma}\otimes F^\ast\c{L}_{\pi})[-n],\]
where $\c K_\gamma $ is the Kummer isocrystal on $\s T^n$ associated to the parameter $\gamma$, $\c L_\pi$ is the Dwork isocrystal on $\s A^1$ associated to $\pi$, $\rho_1^\ast\c{K}_{\gamma}\,\otimes\, F^\ast\c{L}_{\pi}$ is an overconvergent isocrystal on $\b{T}^n_{\kappa}\times\b{A}^N_{\kappa}$.
It can be realized as a coherent $\c{D}_{\s{P}^n\times \s{P}^N,\b{Q}}^{\dagger}(^\dagger(\b{P}^n_{\kappa}\times\b{P}_{\kappa}^N-\b{T}^n_{\kappa}\times\b{A}_{\kappa}^N))$ module by \cite[Theorem 6.2.8]{Caro19}. We have the following main theorem.

\begin{thm}\label{main}
Suppose $\gamma=(\gamma_1,\dots,\gamma_n)$ is $p$-adically semi-nonresonant (defined in \Cref{nonresonance}). The natural morphism (constructed in \Cref{canon})
\[C^\bullet_{\gamma}(\c L^\dagger)\to \m{Hyp}_\gamma\]
is a quasi-isomorphism.
\end{thm}

A closely related algebraic problem was first studied by Schulze--Walther \cite{SW09}.
They introduced the non-strongly resonant condition on the parameter $\gamma$ and proved that, under this condition, the GKZ system can be represented as the direct image of a twisted $D$-module on the algebraic torus.
This perspective was further developed from different points of view by Adolphson--Sperber \cite{AS10} and by Fu \cite{Fu26}. 

The $p$-adic situation is more delicate than its algebraic counterpart.
In the one-dimensional case of generalized hypergeometric equations in the sense of Katz, Miyatani \cite{Miyatani22} proved an analogous comparison result under a $p$-adic non-Liouville condition on the parameters.
The purpose of the present paper is to extend this type of result to general GKZ hypergeometric systems.

\subsection{More results on $p$-adic GKZ hypergeometric $\c{D}$-modules and complexes}

We establish further cohomological and arithmetic properties of the $p$-adic GKZ hypergeometric complex by studying it directly, which extends the corresponding results of
Fu--Li--Wan--Zhang \cite{padictorus} to more general parameters $\gamma$. Then the corresponding results for $\operatorname{Hyp}_\gamma$ follow from the comparison theorem \ref{main} and its invariance under integral shifts of the parameter.

Let $\Delta$ be the convex hull of $\{0,\mathbf w_1, \ldots,\mathbf w_N\}$ in $\mathbb R^n$, $U_0$ the Zariski open subset of $\mathbb A^N_\kappa$ consisting of those points $\bar{\mathbf a}=(\bar a_1, \ldots,\bar a_N)$ so that 
$F(\bar {\mathbf a}, \mathbf t)=\sum_{j=1}^N \bar a_j \mathbf  t^{\mathbf w_j}$ is \emph{non-degenerate} in the sense that for any face $\tau$ of $\Delta$ not containing the origin, the system of equations 
$$\frac{\partial}{\partial t_1}F_\tau(\bar{\mathbf a},\mathbf t)=\cdots =\frac{\partial}{\partial t_n}F_\tau(\bar{\mathbf a},\mathbf t)=0$$ has no solution in $(\bar \kappa^\ast)^n$, 
where $F_\tau(\bar{\mathbf a},\mathbf t)=\sum_{\mathbf w_j\in \tau}\bar a_j \mathbf t^{\mathbf w_j}$. Let $\operatorname{vol}(\Delta)$ be the volume of $\Delta$ in $\b{R}^n$, where the volume is normalized so that $[0,1]^n$ has volume one. In this article, we prove the following theorem.
\begin{thm}\label{Dmodule}
For $\gamma=(\gamma_1,\ldots,\gamma_n)\in \b Z_p^n$, let $\Sigma_{\gamma}$ be the subgroup $\sum_{i=1}^n\b{Z}\gamma_i\subseteq\b{Z}_p$. We have 
   \begin{itemize}
\item[(i)] If all the facet values $H_b(\gamma)$ are $p$-adically non-Liouville (defined in \Cref{nonresonance}), then $\operatorname{Hyp}_\gamma$ is a bounded coherent complex and $H^k(\operatorname{Hyp}_\gamma)=0$ for $k\ne n$. If, moreover, $\Sigma_\gamma$ contains no $p$-adic Liouville numbers, then $H^n(\operatorname{Hyp}_\gamma)$ is overholonomic.

\item[(ii)] Assume $\gamma_i\in \frac{\b Z}{q-1}$. Then  $H^n(\m{Hyp}_\gamma)$ defines an overconvergent $F$-isocrystal of rank $n!\operatorname{vol}(\Delta)$ on $U_0$. More precisely, for any divisor $T$ containing $\mathbb P_\kappa^N-U_0$, $(^\dagger T)H^n(\m{Hyp}_\gamma)$ is a convergent isocrystal of rank $n!\operatorname{vol}(\Delta)$ on $\mathbb P_\kappa^N-T$ overconvergent along $T$.
\end{itemize} 
\end{thm}

\begin{thm}\label{complex}
For $\gamma=(\gamma_1,\ldots,\gamma_n)\in \b Z_p^n$, let $\Sigma_{\gamma}$ be the subgroup $\sum_{i=1}^n\b{Z}\gamma_i\subseteq\b{Z}_p$. We have 
   \begin{itemize}
  \item[(i)] The complex $C_{\gamma}^\bullet(\c L^\dagger)\in D_{\m{coh}}^b(\c{D}_{\s{P}^N,\b{Q}}^{\dagger}(\infty_N))$ is concentrated in degree $n$. If $\Sigma_{\gamma}$ does not have $p$-adic Liouville numbers, then it is overholonomic.
  \item[(ii)] If $\gamma_i\in \frac{\b Z}{q-1}$, then $H^n(C_{\gamma}^\bullet(\c L^\dagger))$ admits a Frobenius structure. Moreover, if $\gamma\in-\delta$, then the Frobenius structure can be defined by the Dwork operator, where $\delta=\mathbb R_{\geq0}\Delta$.
  \item[(iii)] Assume $\gamma_i\in \frac{\b Z}{q-1}$. Then $H^n(C_{\gamma}^\bullet(\c L^\dagger))$ defines an overconvergent $F$-isocrystal of rank $n!\operatorname{vol}(\Delta)$ on $U_0$. More precisely, for any divisor $T$ containing $\mathbb P_\kappa^N-U_0$, $(^\dagger T)H^n(C_{\gamma}^\bullet(\c L^\dagger))$ is a convergent isocrystal of rank $n!\operatorname{vol}(\Delta)$ on $\mathbb P_\kappa^N-T$ overconvergent along $T$.

   \end{itemize}
    
\end{thm}

Both theorems no longer involve the semi-nonresonance condition on the parameter $\gamma$.
From the arithmetic point of view, the $p$-adic GKZ hypergeometric complex is closely related to finite-field exponential sums.
Let
\[ F(\bf{x},\bf{t}) = \sum_{j=1}^N x_j t_1^{w_{1j}}\cdots t_n^{w_{nj}}\]
be the universal Laurent polynomial.
For a point $\bf{a}=(a_1,\ldots,a_N)\in \mathbb A^N(\mathbb F_q)$, a non-trivial additive character $\psi:\mathbb F_q\to \overline{\mathbb Q}_p^\ast$, and multiplicative characters $\chi_1,\ldots,\chi_n:\mathbb F_q^\ast\to \overline{\mathbb Q}_p^\ast$,
the associated finite-field GKZ hypergeometric function is the twisted exponential sum
\[S(F(\bf{x},\bf{t})) = \sum_{t_1,\ldots,t_n\in \mathbb F_q^\ast} \chi_1(t_1)\cdots \chi_n(t_n)\, \psi\left( \sum_{j=1}^N a_j t_1^{w_{1j}}\cdots t_n^{w_{nj}}\right).\]
When $(q-1)\gamma\in\mathbb Z^n$, the class of $\gamma$ modulo $\mathbb Z^n$ determines the multiplicative characters $\chi_i$ through the Teichmüller character. Under the additional assumption $\gamma\in-\delta$, Fu–Li–Wan–Zhang \cite{padictorus} express these exponential sums as alternating traces of the normalized Dwork operator on the fiber complex.

This relates the GKZ hypergeometric complex to the Dwork cohomology of toric exponential sums.
In the work of Adolphson and Sperber, such exponential sums are studied by constructing a $p$-adic Banach space with basis indexed by the lattice points in the cone generated by the Newton polytope $\Delta$, and by considering the associated Dwork complex.
Our comparison theorem shows that this construction is compatible with the functorial $p$-adic GKZ $D$-module construction.
In particular, it shows that the resulting Dwork cohomology is independent of the choice of the cone under the $p$-adic semi-nonresonance condition.

The existence of Frobenius structures is one of the central arithmetic features of $p$-adic hypergeometric equations.
As emphasized by Kedlaya, hypergeometric differential equations are basic examples of Picard--Fuchs equations and rigid local systems, and their $p$-adic Frobenius structures encode the variation of zeta functions of the corresponding motives over finite fields \cite{Kedlaya22}.
This idea goes back to Dwork's construction of Frobenius actions on $p$-adic hypergeometric cohomology and is closely related to Katz's interpretation of hypergeometric sums as trace functions of sheaves \cite{Dwork90,Katz90}.
In the framework of arithmetic $D$-modules, related $p$-adic hypergeometric $D$-modules and $F$-isocrystal structures have also been studied by Miyatani under suitable non-Liouville conditions on the parameters \cite{Miyatani20,Miyatani22}.

From this point of view, a $p$-adic GKZ system should not only be a quotient by differential operators;
it should also be realized as a Frobenius-compatible cohomological object whose trace functions recover finite-field GKZ hypergeometric sums.
Fu--Li--Wan--Zhang constructed such a Frobenius structure on the $p$-adic GKZ hypergeometric complex under certain restrictions on the parameter and the Newton polytope \cite{padictorus}.
In this paper, we prove that the $p$-adic GKZ hypergeometric complex carries a Frobenius structure for all parameters
\[\gamma\in \big(\frac{\b Z}{q-1}\big)^n,\]
or equivalently for arbitrary multiplicative characters of$(\mathbb F_q^\ast)^n$.
Thus the GKZ hypergeometric complex gives a natural arithmetic $D$-module-theoretic realization of finite-field GKZ hypergeometric functions beyond the previously known cases.

The paper is organized as follows. In \Cref{basicsetting}, we recall the relevant operations on arithmetic $\mathcal D$-modules, define the hypergeometric object $\operatorname{Hyp}_\gamma$, and describe its realization using toric compactifications.
In \Cref{section2}, we construct the $p$-adic GKZ hypergeometric complex, give relative de Rham descriptions of the two objects, and define the canonical comparison morphism. 
In \Cref{Comparison}, we prove the comparison theorem by successively removing the inequalities defining the exponent cone, using exponential conjugation and a contracting homotopy. We also give an example illustrating the role of the non-Liouville condition.
Finally, in \Cref{moreon}, we prove concentration in degree $n$, establish overholonomicity under the stated hypotheses, construct Frobenius structures, and determine the rank on the non-degenerate locus. We then deduce the corresponding properties of $\operatorname{Hyp}_\gamma$.

\section{Arithmetic $\c{D}$-modules}\label{basicsetting}
In this section, we recall some basic definitions and properties of Berthelot's arithmetic $\mathcal D$-modules.
We then use them to construct the $p$-adic GKZ hypergeometric $\mathcal D$-module.

We fix $K$, $R$, and $\pi$ as in the introduction, and normalize $|p|=p^{-1}$. Throughout the paper, subscripts are used to indicate base change.
For an $R$-scheme $X_R$, we denote by $X_\kappa$ and $X_K$ the special fiber and the generic fiber of $X_R$, $\s X$ its $\varpi$-adic formal completion along the special fiber, $X_K^{\m{an}}$ its analytification.
When no confusion can arise, we use the same notation for a morphism and its base changes to different base fields.
All the $\c{D}$-modules considered in this paper are left $\c{D}$-modules.

\begin{defn}
    $(1)$ A \textbf{d-couple} is a pair $(\s{X}, T)$, where $\s{X}$ is a smooth formal scheme over $\operatorname{Spf}(R)$ and where $T$ is a divisor of the special fiber of $\s{X}$ (the empty set is also regarded as a divisor). If a $\kappa$-variety $U$ is the special fiber of $\s{X}-T$, we say that $(\s{X}, T)$ \textbf{realizes} $U$.

    $(2)$ A \textbf{morphism of d-couples} $f:\left(\s{X}^{\prime}, T^{\prime}\right) \rightarrow(\s{X}, T)$ is a morphism $\bar{f}: \s{X}^{\prime} \rightarrow \s{X}$ such that $\bar{f}\left(\s{X}^{\prime} \backslash T^{\prime}\right) \subseteq \s{X} \backslash T$ and that $\bar{f}^{-1}(T)$ is a divisor (or empty). We say that $f$ \textbf{realizes} the morphism $f_0: U^{\prime} \rightarrow U$ of $\kappa$-varieties if $\left(\s{X}^{\prime}, T^{\prime}\right)$ (resp. $(\s{X}, T)$) realizes $U^{\prime}$ (resp. $U$) and if $f$ induces $f_0$.
\end{defn}

\subsection{Definitions of $\c{O}_{\s{X},\b{Q}}(^\dagger T)$ and $\c{D}_{\s{X},\b{Q}}^{\dagger}(^\dagger T)$}
For each d-couple $(\s{X}, T)$,  we recall the definition of $\c{O}_{\s{X},\b{Q}}(^\dagger T)$ and $\c{D}_{\s{X},\b{Q}}^{\dagger}(^\dagger T)$ in \cite[4.2.4 and 4.2.5]{Berthelot1}. Let $\s{U} \subset \s{X}$ be an affine open subset and let $h \in \Gamma(\s{U}, \c{O}_{\s{X}})$ so that the divisor $T$ of $X_{\kappa}$ is given by $h \equiv 0 \bmod (\varpi)$. Let $U_i$ and $h_i \in \Gamma(U_i, \c{O}_{U_i})$ be the reductions of $\s{U}$ and $h$ modulo $\varpi^{i+1}$, respectively. For any $m \geq 0$, define
\begin{align*}
    \c{B}_{X_i}^{(m)}(T)|_{U_i} & =\c{O}_{U_i}[t]/(h_i^{p^{m+1}} t-p),
    \\ \widehat{\c{B}}_{\s{X}}^{(m)}(T)|_{\s{U}} & =\varprojlim_{i} \,\,\c{B}_{X_i}^{(m)}(T)|_{U_i}=\c{O}_{\s{U}}\{t\} /(h^{p^{m+1}} t-p).
\end{align*}
$\widehat{\c{B}}_{\s{X}}^{(m)}(T)$ is a $p$-adically complete $\c{O}_{\s{X}}$-algebra depending only on $\s{X}$ and $T$. Define the sheaf of functions on $\s{X}$ with overconvergent singularities along $T$ by
\[\c{O}_{\s{X}, \mathbb{Q}}(^{\dagger} T)=\varinjlim_m\,\,\widehat{\c{B}}_{\s{X}}^{(m)}(T) \otimes_{\mathbb{Z}} \mathbb{Q} .\]
Denote by $\c{D}_{X_i}^{(m)}$ the sheaf of differential operators of level $m$ on $X_i$. If $x_1, \ldots, x_N$ are local coordinates on $U_i$ and $\partial_j=\partial / \partial x_j$ for $1 \leq j \leq N$, we have
\[\c{D}_{X_i}^{(m)}(U_i)=\Big\{\sum_{\mathbf{k}} \mathbf{q}_{\mathbf{k}}^{(m)}!a_{\mathbf{k}} \partial^{[\mathbf{k}]}\,\Big|\, a_{\mathbf{k}} \in \Gamma\left(U_i, \c{O}_{X_i}\right), a_{\mathbf{k}} \neq 0 \text { for only finitely many } \mathbf{k}\Big\},\]
where $\partial^{[\mathbf{k}]}=\frac{1}{\mathbf{k}!} \partial_1^{k_1} \ldots \partial_N^{k_N}$, and $\mathbf{q}_{\mathbf{k}}^{(m)}$ is determined by the expression $\mathbf{k}=p^m \mathbf{q}_{\mathbf{k}}^{(m)}+ \mathbf{r}_{\mathbf{k}}^{(m)}$ with $0 \leq r_{\mathbf{k}, j}^{(m)}<p^m$ for all $j$. The sheaf of differential operators of level $m$ on $\s{X}$ is defined by
\[\widehat{\c{D}}_{\s{X}}^{(m)}=\varprojlim_i \c{D}_{X_i}^{(m)}, \quad \widehat{\c{D}}_{\s{X}, \mathbb{Q}}^{(m)}=\widehat{\c{D}}_{\s{X}}^{(m)} \otimes_{\mathbb{Z}} \mathbb{Q} .\]
We have
\[\widehat{\c{D}}_{\s{X}, \mathbb{Q}}^{(m)}(\s{U})=\Big\{\sum_{\mathbf{k}} \mathbf{q}_{\mathbf{k}}^{(m)}!a_{\mathbf{k}} \partial^{[\mathbf{k}]}\,\Big|\, a_{\mathbf{k}} \in \Gamma\left(\s{U}, \c{O}_{\s{X}, \mathbb{Q}}\right), a_{\mathbf{k}} \rightarrow 0 \text { as }|\mathbf{k}| \rightarrow \infty\Big\},\]
where $\c{O}_{\s{X}, \mathbb{Q}}=\c{O}_{\s{X}} \otimes_{\mathbb{Z}} \mathbb{Q}$. 
Set
\[\widehat{\mathcal D}_{\s X}^{(\bullet)}(T):= \bigl(\widehat{\mathcal D}_{\s X}^{(m)}(T)\bigr)_{m\geq 0},
  \qquad\widehat{\mathcal D}_{\s X}^{(m)}(T):= \widehat{\mathcal B}_{\s X}^{(m)}(T) \widehat\otimes_{\mathcal O_{\s X}}\widehat{\mathcal D}_{\s X}^{(m)} .\]
Let $$\c{D}_{\s{X},\b{Q}}^{\dagger}(^\dagger T):=\varinjlim_m\,\,\widehat{\c{D}}_{\s{X}, \mathbb{Q}}^{(m)}(T)$$
be the sheaf of differential operators on $\s X$ overconvergent along $T$.

\subsection{}
Let $f:(\s{X}',T')\to(\s{X},T)$ be a morphism of d-couples and let $f_i:X'_i\to X_i$ be the reduction of $f$ modulo $\varpi^{i+1}$. Then we recall the definitions of cohomological operations on the categories $D^b(\c{D}_{\s{X},\b{Q}}^{\dagger}(^\dagger T))$, which can be found in \cite[1.6]{Carosurhol}, \cite[2.1.3]{Carotensorproduct} and \cite[3.1]{Caro19}.

\subsection{Extraordinary pull-back functors}

Following \cite[1.10]{Carosurhol}, we define the extraordinary pull-back functor $f^!:D_{\m{coh }}^b(\c{D}_{\s{X}, \b{Q}}^{\dagger}(^{\dagger} T))\to D^b(\c{D}_{\s{X}^{\prime}, \b{Q}}^{\dagger}(^{\dagger} T'))$ by the following formula 
\[\c{E}\mapsto \c{D}_{\s{X}'\to\s{X}, \b{Q}}^{\dagger}(^{\dagger} T',T)\otimes^L_{f^{-1}\c{D}_{\s{X}, \b{Q}}^{\dagger}(^{\dagger} T)} f^{-1}\c{E}[d_{\s{X}'/\s{X}}],\]
where \[\c{D}_{\s{X}'\to\s{X}, \b{Q}}^{\dagger}(^{\dagger}T',  T):=\b{Q}\otimes_{\b{Z}}\varinjlim_m\widehat{\mathcal D}^{(m)}_{\s X'\to\s X}(T',T),\] $\widehat{\mathcal D}^{(m)}_{\s X'\to\s X}(T',T):=\varprojlim_i\,\c{B}_{X'_i}^{(m)}(T')\otimes_{f_i^{-1}(\c{O}_{X_{i}})}f_i^{-1}(\c{D}_{X_i}^{(m)})$
is the usual finite level transfer bimodule, the superscript $L$ means the left derived functor of tensor product, $d_{\s{X}'},d_{\s{X}}$ are the dimensions of $\s{X}',\s{X}$, 
$d_{\mathscr X'/\mathscr X}:=d_{\mathscr X'}-d_{\mathscr X}.$
If $\bar{f}$ is smooth, or if $\bar{f}$ induces an open immersion $\s{X}^{\prime} \backslash T^{\prime}\hookrightarrow \s{X} \backslash T$, then the essential image of $f^{!}$ lies in $D_{\m{coh }}^b(\c{D}_{\s{X}^{\prime}, \b{Q}}^{\dagger}(^{\dagger} T'))$.

If $\s{X}^{\prime}=\s{X}$ and if $\bar{f}$ is the identity morphism on $\s{X}$, the pull-back functor $f^!$ is also denoted by
\begin{eqnarray*}
(^\dagger T', T): 
D_{\text{coh}}^b(\mathcal D_{\mathscr X, \mathbb Q}^\dagger(^\dagger T))\to D_{\text{coh}}^b (\mathcal D_{\mathscr X, \mathbb Q}^\dagger(^\dagger T')),\quad
(^\dagger T', T)(\mathcal E):=\mathcal D_{\mathscr X, \mathbb Q}^\dagger(^\dagger T')\otimes^L_{\mathcal D_{\mathscr X, \mathbb Q}^\dagger(^\dagger T)}\mathcal E.
\end{eqnarray*}
Here we can omit the symbol $L$ since the morphism 
$\mathcal D_{\mathscr X, \mathbb Q}^\dagger(^\dagger T)\to \mathcal D_{\mathscr X, \mathbb Q}^\dagger(^\dagger T')$ is flat (cf. \cite[4.3.10-11]{Berthelot1}).

Let $f^{\prime}:\left(\s{X}^{\prime \prime}, T^{\prime \prime}\right) \rightarrow(\s{X}^{\prime}, T^{\prime})$ be another morphism of d-couples and let $\c{E}$ be an object of $D_{\m{coh }}^{b}(\c{D}_{\s{X}, \b{Q}}^{\dagger}(^{\dagger} T))$. Then, as long as $f^{!} \c{E}$ belongs to $D_{\m{coh }}^{b}(\c{D}_{\s{X}^{\prime}, \b{Q}}^{\dagger}(^{\dagger} T'))$, we have a natural isomorphism $f'^! f^!(\c{E})\cong \left(f \circ f^{\prime}\right)^!(\c{E})$. If we consider overholonomic complexes in the sense of \cite[3.1]{Carosurhol}, the functor $f^!$ preserves the property of being overholonomic by \cite[3.8]{Carosurhol}; hence, we have the functor \[f^{!}: D_{\m{ovhol}}^b(\c{D}_{\s{X}, \b{Q}}^{\dagger}(^{\dagger} T)) \rightarrow D^b_{\m{ovhol}}(\c{D}_{\s{X}^{\prime}, \b{Q}}^{\dagger}(^{\dagger} T')).\]

\subsection{Ordinary push-forward functors}
Similarly, following \cite[1.10]{Carosurhol}, we define the ordinary push-forward functor $f_+:D_{\m{coh }}^b(\c{D}_{\s{X}', \b{Q}}^{\dagger}(^{\dagger} T'))\to D^b(\c{D}_{\s{X}, \b{Q}}^{\dagger}(^{\dagger} T))$
by the following formula   
\[\c{E}'\mapsto Rf_*\big(\c{D}_{\s{X}\leftarrow\s{X}', \b{Q}}^{\dagger}(T,T')\otimes^L_{\c{D}_{\s{X}', \b{Q}}^{\dagger}(^{\dagger} T')} \c{E}'\big),\] where $Rf_*$ is the right derived functor of $f_*$, \[\c{D}_{\s{X}\leftarrow\s{X}', \b{Q}}^{\dagger}(T,T'):=\b{Q}\otimes_{\b{Z}}\varinjlim_m\,\widehat{\mathcal D}^{(m)}_{\s X\leftarrow\mathscr{X}'}(T,T'),\] 
and $\widehat{\mathcal D}^{(m)}_{\mathscr{X}\leftarrow\mathscr{X}'}(T,T')$
is the usual finite level right transfer bimodule.
More precisely, it is given by
\[\widehat{\mathcal D}^{(m)}_{\s X\leftarrow\mathscr{X}'}(T,T') := \varprojlim_i\,\c{B}_{X'_i}^{(m)}(T')\otimes_{\c{O}_{X'_i}}\omega_{X'_i/R_i}\otimes_{f_i^{-1}(\c{O}_{X_{i}})}f_i^{-1}(\c{D}_{X_i}^{(m)}\otimes_{\c{O}_{X_i}}\omega_{X_i/R_i}^{-1}),\]
where $R_i:=R/\varpi^{i+1}$,  $\omega_{X_i/R_i},\omega_{X'_i/R_i}$ are the canonical bundles of $X_i,X'_i$ over $R_i$, respectively. 
If $\bar{f}$ is proper and $T^{\prime}=\bar{f}^{-1}(T)$, then the essential image of $f_{+}$ lies in $D_{\m{coh }}^b(\c{D}_{\s{X}, \mathbb{Q}}^{\dagger}(^{\dagger} T))$.
Let $f^{\prime}:\left(\s{X}^{\prime \prime}, T^{\prime \prime}\right) \rightarrow\left(\s{X}^{\prime}, T^{\prime}\right)$ be another morphism of d-couples, and let $\c{E}^{\prime \prime}$ be an object of $D_{\m{coh }}^b(\c{D}_{\s{X}^{\prime \prime}, \b{Q}}^{\dagger}\left({ }^{\dagger} T^{\prime \prime}\right))$. Then, as long as $f_{+}^{\prime} (\c{E}^{\prime \prime})$ is an object of $D_{\m{coh }}^{b}(\c{D}_{\s{X}^{\prime}, \b{Q}}^{\dagger}(^{\dagger} T'))$, we have a natural isomorphism $f_{+} f_{+}^{\prime} (\c{E}^{\prime \prime}) \cong\left(f \circ f^{\prime}\right)_{+} (\c{E}^{\prime \prime})$.

If $\bar{f}$ is proper and we consider overholonomic complexes, the functor $f_+$ preserves the property of being overholonomic by \cite[3.9]{Carosurhol}; hence, we have the functor \[f_+: D_{\m{ovhol}}^b(\c{D}_{\s{X}', \b{Q}}^{\dagger}(^{\dagger} T')) \rightarrow D^b_{\m{ovhol}}(\c{D}_{\s{X}, \b{Q}}^{\dagger}(^{\dagger} T)).\]

If $\s{X}^{\prime}=\s{X}$ and if $\bar{f}$ is the identity morphism on $\s{X}$ (thus $f$ represents an open immersion), then $f_{+}$ is obtained by considering the complex of $\c{D}_{\s{X}, \b{Q}}^{\dagger}(^{\dagger} T')$-modules as a complex of $\c{D}_{\s{X}, \b{Q}}^{\dagger}(^{\dagger} T)$-modules via the inclusion $\c{D}_{\s{X}, \b{Q}}^{\dagger}(^{\dagger} T) \hookrightarrow \c{D}_{\s{X}, \b{Q}}^{\dagger}(^{\dagger} T')$.

\subsection{The functors $\m{sp}_+$}\label{sp}

Recall the construction in \cite[2.5]{Carospfunctor}. For any d-couple $(\s{X},T)$ with $\s{X}$ proper, we have a functor \[\m{sp}_+:\m{Isoc}^{\dagger}(U,X/K)\to \m{Coh}(\c{D}_{\s{X},\b{Q}}^{\dagger}(^\dagger T)),\] where $X$ is the special fiber of $\s{X}$, $U:=X-T$.
By \cite[Theorem 4.2.2]{Caro2011PleineFS}, $\m{sp}_+$ is fully faithful and induces an equivalence on its essential image.

$\m{sp}_+$ commutes with cohomological operators. 

Let $f:(\s{X}',T')\to (\s{X},T)$ be a morphism between d-couples. By \cite[Proposition 4.2.4]{Caro2011PleineFS}, we have the canonical isomorphism
\[\m{sp}_+f^*\simeq f^![-d_{\s{X}'}+d_{\s{X}}]\m{sp}_+\] on the category $\m{Isoc}^{\dagger}(U,X/K)$, where $d_{\s{X}'},d_{\s{X}}$ are the dimensions of $\s{X}',\s{X}$, respectively. 

Let $U':=X'-T'$, and let $\s{U},\s{U}'$ be the open formal subschemes of $\s{X},\s{X}'$ with special fibers $U,U'$, respectively. If we assume that $f:\s{U}'\to\s{U}$ is finite \'{e}tale, then there exists a functor $f_*:\m{Isoc}^{\dagger}(U',X'/K)\to \m{Isoc}^{\dagger}(U,X/K)$ (cf. \cite[Theorem 2.6.3]{Tsuzuki99}). Moreover, we have the canonical isomorphism \[\m{sp}_+f_*\simeq f_+\m{sp}_+\] by \cite[Lemma 7.2.1]{Caro06}.

For the twisted tensor functors, we have \[\m{sp}_+(-\otimes-)\simeq \m{sp}_+(-){\otimes}_{\c{O}_{\s{X}, \b{Q}}(^{\dagger} T)}^{\dagger,L}\m{sp}_+(-)\] by \cite[Proposition 3.1.8]{Carotensorproduct}.

\subsection{$p$-adic GKZ hypergeometric $\c D$-modules}\label{hyp}
Now we use the language of d-couples to construct $p$-adic GKZ hypergeometric $\c D$-modules $\m{Hyp}_\gamma$.
Let $\gamma=(\gamma_1,\ldots, \gamma_n)\in \b Z_p^n$. For each $i$,
let $\c K_{\gamma_i}$ be the Kummer isocrystal on $\b T_{\kappa}^1/K$ with parameter $\gamma_i$ \cite[2.2.15]{Berrigid}. Define
$$\c K_\gamma=\c K_{\gamma_1}\boxtimes\cdots\boxtimes\c K_{\gamma_n} .$$
It is an overconvergent isocrystal on $\b T_{\kappa}^n/K$. Let $H_T=\b P_{\kappa}^n-\b T_{\kappa}^n$ be the toric boundary. Then $\text{sp}_+\c K_\gamma$ is  a coherent $\c D_{\s P^n,\b Q}^\dagger(^\dagger H_T)$-module by section \ref{sp}.
Recall that $\rho_1:\b{T}_{\kappa}^n\times \b{A}_{\kappa}^N\to\b{T}_{\kappa}^n$ and $\rho_2:\b{T}_{\kappa}^n\times \b{A}_{\kappa}^N\to\b{A}_{\kappa}^N$ are the projections, and let $F$ be the morphism \[F:\b{T}_{\kappa}^n\times \b{A}_{\kappa}^N\to \b{A}_{\kappa}^1,\quad (t_1,\cdots,t_n,x_1,\cdots,x_N)\mapsto\sum_{j=1}^Nx_jt_1^{w_{1j}}\cdots t_n^{w_{nj}}.\]
Set
\[\infty_N:=\mathbb P^N_{\kappa}\setminus \mathbb A^N_{\kappa}\]
and
\[Z:=\bigl(H_T\times \mathbb P^N_{\kappa}\bigr)\cup    \bigl(\mathbb P^n_{\kappa}\times \infty_N\bigr) \subset  \mathbb P^n_{\kappa}\times \mathbb P^N_{\kappa}.\]
Then the d-couple $(\mathscr{P}^n\times \mathscr{P}^N,Z)$ realizes $\mathbb T^n_{\kappa}\times \mathbb A^N_{\kappa}$.
We use the same symbols $\rho_1,\rho_2$ for the corresponding morphisms of d-couples.
In particular, \[\rho_2:(\mathscr{P}^n\times \mathscr{P}^N,Z)\longrightarrow (\mathscr{P}^N,\infty_N)\]
is a morphism of d-couples realizing the second projection.
By \cite[2.3.6]{Berrigid}, the tensor product $\rho_1^\ast \c K_\gamma\otimes F^\ast \c L_\pi$ is an overconvergent isocrystal on $\b T^n_{\kappa}\times \b A_{\kappa}^N$, and the realization $\text{sp}_+(\rho_1^\ast \c K_\gamma\otimes F^\ast \c L_\pi)$ is a coherent $\c{D}_{\s P^n\times \s{P}^N,\b{Q}}^{\dagger}(^\dagger Z)$-module. 
We define the $p$-adic hypergeometric $\c D$-module with parameter $\gamma$ by
\[\m{Hyp}_{\gamma}:=\rho_{2,+}\bigl(\text{sp}_+(\rho_1^\ast \c K_\gamma\otimes F^\ast \c L_\pi)\bigr)[-n].\]

\subsection{Toric compactifications}
We shall also use another realization of $\mathbb T^n_\kappa$ by a toric compactification. 
Let $\Delta$ be the convex hull of $\{0, \mathbf{w}_1,\ldots,\mathbf{w}_N\}$ in $\b R^n$. We assume that $\dim\Delta=n$.
Let $\Sigma$ be the normal fan defined by the polytope $\Delta$. Set
$X_R(\Sigma)$ to be the proper toric $R$-scheme defined by $\Sigma$.

Choose a regular complete fan $\Sigma'$ that refines $\Sigma$. Then $X_R(\Sigma')$ is smooth and proper. 
Replacing $\Sigma'$ by a further regular refinement, we may assume that there exists a birational morphism \[v:X_R(\Sigma') \to \mathbb P^n_R.\]
Denote by $X_\kappa(\Sigma), X_\kappa(\Sigma')$ (resp. $\s X(\Sigma), \s X(\Sigma')$) the special fibers (resp. formal completions) of $X_R(\Sigma), X_R(\Sigma')$ over $\kappa$ (resp. $R$).
Let $$H_{\Sigma'}=X_\kappa(\Sigma')-\b T_\kappa^n, \ \ Z'=(H_{\Sigma'}\times \b P_\kappa^N)\cup(X_\kappa(\Sigma')\times \infty_N).$$
Then the d-couple $(\s X(\Sigma'),H_{\Sigma'})$ also realizes $\b T^n_\kappa$, and  $\bigl(\s X(\Sigma')\times \s P^N, Z'\bigr)$ realizes $\b T^n_{\kappa}\times \b A^N_\kappa$. 
The overconvergent isocrystal $\rho_1^\ast \c K_\gamma\otimes F^\ast \c L_\pi$ on $\b T^n_{\kappa}\times \b A^N_\kappa$ is also realized by a coherent $\c D^\dagger_{\s X(\Sigma')\times \s P^N,\b{Q}}(^\dagger Z')$-module $\text{sp}_+(\rho_1^\ast \c K_\gamma\otimes F^\ast \c L_\pi)$. Let 
\[\rho_1':\bigl(\s X(\Sigma')\times \s P^N, Z'\bigr)\to (\s X(\Sigma'),H_{\Sigma'}), \ \rho_2':\bigl(\s X(\Sigma')\times \s P^N, Z'\bigr)\to (\s P^N,\infty_N)\]
be the projections. Then we can also define
\[\m{Hyp}_{\gamma}':=\rho_{2,+}'\m{sp}_+(\rho_1^\ast\c{K}_{\gamma}\otimes F^\ast\c{L}_{\pi})[-n].\]
The canonical isomorphism between $\m{Hyp}_{\gamma}'$ and $\m{Hyp}_{\gamma}$ follows from the commuting diagram
\[\begin{tikzcd}
    \text{Isoc}^{\dagger}(\b T^n_{\kappa}\times \b A^N_{\kappa},X_{\kappa}(\Sigma')\times \b{P}_{\kappa}^N/K) \ar[r,"\m{sp}_+"]\ar[d,"\m{id}_*","\cong"']  & \m{Coh}(\c D^\dagger_{\s X(\Sigma')\times \s P^N,\b{Q}}(^\dagger Z'))\ar[d,"(v\times\m{id})_+"]\ar[r,"\rho'_{2,+}"] & D^b(\c D_{\s P^N,\b{Q}}^\dagger(\infty_N))
    \\ \text{Isoc}^{\dagger}(\b T^n_{\kappa}\times \b A^N_{\kappa},\b{P}^n_{\kappa}\times \b{P}_{\kappa}^N/K)\ar[r,"\m{sp}_+"'] & \m{Coh}(\c D^\dagger_{\s P^n\times \s P^N,\b{Q}}(^\dagger Z))\ar[ru,"\rho_{2,+}"'] & ,
\end{tikzcd}\]
where the left square commutes by \cite[Lemma 7.2.1]{Caro06}.

On the other hand, the lattice points $0,\mathbf w_1,\ldots,\mathbf w_N \in \Delta$ define, by \cite[Lemma 6.0.28]{coxtoric}, a morphism $X_R(\Sigma) \longrightarrow \mathbb P_R^N.$
Let
\[X_{\delta,R}:=\operatorname{Spec} R[\delta\cap \mathbb Z^n]. \]
The restriction of this morphism to the affine open subset $X_{\delta,R}\subset X_R(\Sigma)$ factors through the standard affine chart $\mathbb A_R^N=\{x_0\neq 0\}\subset \mathbb P_R^N.$
Moreover, we have the following Cartesian diagram:
\[\begin{tikzcd}
X'_{\delta,R}:=X_R(\Sigma')\times_{X_R(\Sigma)} X_{\delta,R} \arrow[r] \arrow[d] & X_{\delta,R}\arrow[r] \arrow[d] & \mathbb A_R^N \arrow[d]
\\ X_R(\Sigma') \arrow[r] & X_R(\Sigma) \arrow[r] & \mathbb P_R^N, \end{tikzcd} \]
where $X_{\delta,R}\to \b A^N_R$ is the morphism defined by the $R$-algebra homomorphism 
$$R[x_1, \ldots, x_N]\to R[\delta\cap \b Z^n],\quad x_j\mapsto \mathbf t^{\mathbf w_j}.$$

\section{$p$-adic GKZ hypergeometric complex}\label{section2}

\subsection{Construction of $C_\gamma^{\bullet}(\c L^\dagger)$}\label{C}
For any $\mathbf u=(u_1, \ldots, u_N)\in\mathbb Z_{\geq 0}^N$ and $\mathbf v=(v_1, \ldots, v_n)\in \mathbb Z^n$, 
write $$\mathbf x^{\mathbf u}=x_1^{u_1}\cdots x_N^{u_N},\quad \mathbf t^{\mathbf v}=t_1^{v_1}\cdots t_n^{v_n}, \quad \vert\mathbf u\vert =u_1+\cdots+ u_N, \quad \vert\mathbf v\vert =|v_1|+\cdots+|v_n|.$$ 
A weakly complete algebra $L^\dagger$ is defined by
\[L^{\dagger}:=\Big\{\sum_{\mathbf{u}\in\b{Z}^N_{\geq 0},\mathbf{v}\in \delta\cap \b Z^n}a_{\mathbf{u},\mathbf{v}}\mathbf{x}^{\mathbf{u}}\bf t^{\mathbf{v}}\,\Big|\,a_{\mathbf{u},\mathbf{v}}\in K,\,\exists(c>0,0<\eta<1),\|a_{\mathbf{u},\mathbf{v}}\|\leq c\eta^{|\mathbf{u}|+|\mathbf{v}|}\Big\}.\]
The $p$-adic GKZ hypergeometric complex is the logarithmic twisted de Rham complex $C_\gamma^{\bullet}(L^\dagger)$ defined as follows: For $k\ge 0$, we set

$$C^k( L^\dagger)=
\{\sum_{1\leq i_1<\cdots < i_k\leq n}f_{i_1\ldots i_k}(\mathbf x, \mathbf t)
\frac{\mathrm dt_{i_1}}{t_{i_1}}\wedge \cdots\wedge \frac{\mathrm dt_{i_k}}{t_{i_k}}:\;f_{i_1\ldots i_k}(\mathbf x,\mathbf t)\in  L^\dagger\}$$ 
with differential $d_\gamma: C^k( L^\dagger)\to C^{k+1}( L^\dagger)$ given by
\begin{eqnarray*}
d_{\gamma}(\omega)=\mathrm d_{\mathbf t}\omega + \sum_{i=1}^n \Big(\gamma_i+\sum_{j=1}^N \pi w_{ij}x_j\mathbf t^{\mathbf w_j}\Big)
\frac{\mathrm dt_i}{t_i}\wedge \omega.
\end{eqnarray*}
It is equipped with a connection $\nabla:C_\gamma^\bullet (L^\dagger)\to C_\gamma^\bullet (L^\dagger)\otimes_{K\langle\bf{x}\rangle}\Omega^1_{K\langle\bf{x}\rangle}$ defined by
$$\nabla(\omega)=d_{\bf{x}}\omega+\sum_{j=1}^N \pi\mathbf t^{\mathbf w_j}dx_j\wedge\omega.$$
By \cite[Proposition 1.8]{padictorus}, $C_\gamma^{\bullet}(L^\dagger)$ is a complex of coherent $\Gamma(\s P^N, \mathcal D^\dagger_{\mathscr{P}^N, \mathbb Q}(\infty_N))$-modules. Hence it defines an object in $D^b_{\m{coh}}(\c D^\dagger_{\s P^N,\b Q}(\infty_N)),$ which we denote by $C_\gamma^{\bullet}(\c L^\dagger)$.

Now we describe this object by a relative twisted logarithmic de Rham complex. To describe this construction more precisely, we first fix the following
morphisms and commutative diagram. All schemes are defined over $R$, and the
subscript $R$ is omitted from the notation. Let
\[F(\bf{x},\bf{t})=\sum_{j=1}^N x_j \mathbf{t}^{\mathbf w_j}
  \in R[\delta\cap \mathbb Z^n][x_1,\ldots,x_N].\]
It defines the usual morphism
  $F:\b{T}^n\times \mathbb A^N\longrightarrow \mathbb A^1 .$
The same function
also defines a morphism
$F_\delta:X_\delta\times \mathbb A^N\longrightarrow \mathbb A^1.$ 
We have the following commutative diagram:
\[
\begin{tikzcd}[column sep=large, row sep=large]
& \b T^n \arrow[dl, hook', "j_\delta"'] \arrow[d, hook, "j'_\delta"]
& \b T^n \times \mathbb A^N \arrow[d, hook, "j'_\delta \times \mathrm{id}"] \arrow[dr, "F"] \arrow[l] & 
\\ X_\delta \arrow[d, hook, "j"'] & X'_\delta \arrow[l, "\phi_\delta"] \arrow[d, hook, "j'"] & X'_\delta \times \mathbb A^N \arrow[d, hook, "j' \times j_N"] \arrow[r, "\overline F"] \arrow[l] & \mathbb A^1
\\ X(\Sigma) & X(\Sigma') \arrow[l, "\phi_\Sigma"] & X(\Sigma') \times \mathbb P^N \arrow[l, "\rho_1'"] &
\end{tikzcd}
\qquad\qquad
\begin{tikzcd}[row sep=large]
\mathbb A^N \arrow[d, hook, "j_N"] \\ \mathbb P^N
\end{tikzcd}
\]
where $\overline F=F_\delta\circ (\phi_\delta\times \m{id})$.
The open immersion $j'\times j_N:X'_{\delta, R}\times \b A^N_R\to X_R(\Sigma')\times \b P_R^N$ induces an open immersion of rigid analytic spaces $]X'_{\delta,\kappa}\times \b A_\kappa^N[\to X_K(\Sigma')^{\m{an}}\times \b P^{N,{\m{an}}}_K$, which we still denote by $j'\times j_N$. We have the following lemma:

\begin{lem}\label{Ldagger}\label{global}
    $$L^\dagger=\Gamma\Big(X_K(\Sigma')^{\m{an}}\times \b P^{N,{\m{an}}}_K, (j'\times j_N)^\dagger(\c O_{X_K(\Sigma')^{\m{an}}\times \b P^{N,{\m{an}}}_K})\Big).$$
\end{lem}
\begin{proof}
    By the proof of \cite[Lemma 5.3]{padictorus}, we have an isomorphism
    \[
   (\phi_{\Sigma,K}^{\m{an}}\times\m{id}_{\b{P}^{N,\m{an}}_K})_*\big((j'\times j_N)^\dagger(\c O_{X_K(\Sigma')^{\m{an}}\times \b P^{N,{\m{an}}}_K})\big)\cong (j\times j_N)^\dagger(\c O_{X_K(\Sigma)^{\m{an}}\times \b P^{N,{\m{an}}}_K}).
    \]
    Hence it remains to compute the global sections of the right hand side.
    Recall the closed immersion $X_\delta\to\b A_R^M$ defined in \cite[Section 5.1]{padictorus}. Let $I$ be its defining ideal. 
For $\lambda>1$, let \[ V_\lambda = X_{\delta,K}^{\m{an}}\cap D(0,\lambda^+)^M \subset X_{\delta,K}^{\m{an}}, \qquad U_\lambda=D(0,\lambda^+)^N\subset \b P_K^{N,\m{an}}. \] Then $\{V_\lambda\}_{\lambda>1}$ is a fundamental system of affinoid strict neighborhoods of $]X_{\delta,\kappa}[$ in $X_K(\Sigma)^{\m{an}}$, and $\{U_\lambda\}_{\lambda>1}$ is a fundamental system of affinoid strict neighborhoods of $]\b A_\kappa^N[$ in $\b P_K^{N,\m{an}}$. By the proof of \cite[Lemma 2.1.7]{Berrigid}, the products $V_\lambda\times U_\lambda$ form a fundamental system of affinoid strict neighborhoods of 
\[ ]X_{\delta,\kappa}\times \b A_\kappa^N[ \subset X_K(\Sigma)^{\m{an}}\times \b P_K^{N,\m{an}}. \]

Therefore,
\begin{eqnarray*}
&&\Gamma\left(
X_K(\Sigma)^{\mathrm{an}}\times \mathbb P_K^{N,\mathrm{an}}, (j\times j_N)^\dagger \mathcal O_{X_K(\Sigma)^{\mathrm{an}}\times \mathbb P_K^{N,\mathrm{an}}} \right) \\
&=&
\varinjlim_{\lambda\to 1^+} \Gamma\left( V_\lambda\times  U_\lambda, \mathcal O_{X_K(\Sigma)^{\mathrm{an}}\times \mathbb P_K^{N,\mathrm{an}}} \right) \\
&=&
\varinjlim_{\lambda\to 1^+} \Gamma\left( V_\lambda, \mathcal O_{X_K(\Sigma)^{\mathrm{an}}} \right) \widehat{\otimes}_K \Gamma\left( U_\lambda, \mathcal O_{\mathbb P_K^{N,\mathrm{an}}} \right) \\
&=&
\varinjlim_{\lambda\to 1^+} \left( K\{\lambda^{-1}y_1,\ldots,\lambda^{-1}y_M\} /I K\{\lambda^{-1}y_1,\ldots,\lambda^{-1}y_M\} \right) \widehat{\otimes}_K K\{\lambda^{-1}x_1,\ldots,\lambda^{-1}x_N\} \\
&=&
\varinjlim_{\lambda\to 1^+} K\{\lambda^{-1}y_1,\ldots,\lambda^{-1}y_M, \lambda^{-1}x_1,\ldots,\lambda^{-1}x_N\} /I K\{\lambda^{-1}y_1,\ldots,\lambda^{-1}y_M,  \lambda^{-1}x_1,\ldots,\lambda^{-1}x_N\}  \\
&=&
K\langle y_1,\ldots,y_M,x_1,\ldots,x_N\rangle^\dagger/I K\langle y_1,\ldots,y_M,x_1,\ldots,x_N\rangle^\dagger    \\
&\cong& L^\dagger .
\end{eqnarray*}
The last isomorphism is obtained by applying the proof of
\cite[Lemma 5.2]{padictorus} after adjoining the overconvergent variables
$x_1,\ldots,x_N$. This proves the lemma.
\end{proof}

Let $D'=X_{\delta,\kappa}'-\b T^n_\kappa$. As a resolution of singularities of $X_{\delta,\kappa}$, $X_{\delta,\kappa}'$ is smooth, and $D'$ is a normal crossing divisor.
Hence the sheaf $(j'\times j_N)^\dagger \Omega^k_{X_K(\Sigma')^{\m{an}}\times \b P^{N,{\m{an}}}_K/\b P^{N,{\m{an}}}_K}\big(\m{log} (D'\times \b P^N)\big)$ of relative $k$-forms with logarithmic poles along $D'$ is a free  $(j'\times j_N)^\dagger(\c O_{X_K(\Sigma')^{\m{an}}\times \b P^{N,{\m{an}}}_K})$-module with basis
$$\frac{\mathrm dt_{i_1}}{t_{i_1}}\wedge\cdots \wedge \frac{\mathrm dt_{i_k}}{t_{i_k}}\quad (1\leq i_1<\cdots<i_k\leq n)$$ for each $k$.

Let $\widetilde{\c{K}}_\gamma$ be $j'^\dagger\c O_{X_K(\Sigma')^{\m{an}}}$ endowed with the logarithmic connection
$$\nabla:j'^\dagger\c O_{X_K(\Sigma')^{\m{an}}}\to j'^\dagger \Omega^1_{X_K(\Sigma')^{\m{an}}}(\m{log} D')\quad \quad g\to dg+(\sum_{i=1}^n \frac{\gamma_idt_i}{t_i})g  .$$
Let $\widetilde{\c L}_{\pi,\overline F}$ (resp. $\widetilde{\c L}_{\pi,F}$) be $(j'\times j_N)^\dagger(\c O_{X_K(\Sigma')^{\m{an}}\times \b P^{N,{\m{an}}}_K})$ (resp. $(j'j_\delta'\times j_N)^\dagger(\c O_{X_K(\Sigma')^{\m{an}}\times \b P^{N,{\m{an}}}_K})$) with an integrable connection which realizes $\overline{F}^\ast \c L_\pi$ (resp. $F^\ast \c L_\pi$) . 
Then we have the following complex: 
\[\rho'_{2,\ast}\m{sp}_\ast\Big((j'\times j_N)^\dagger \Omega^\bullet_{X_K(\Sigma')^{\m{an}}\times \b P^{N,{\m{an}}}_K/\b P^{N,{\m{an}}}_K}\big(\m{log} (D'\times \b P^N)\big) \otimes \rho_1'^\ast \widetilde{\c{K}}_\gamma\otimes \widetilde{\c L}_{\pi,\overline F}\Big),\] 
where the tensor products are taken over $(j'\times j_N)^\dagger(\c O_{X_K(\Sigma')^{\m{an}}\times \b P^{N,{\m{an}}}_K})$, $\m{sp}:\b P^{N,{\m{an}}}_K\to \s P^N$ is the specialization morphism. 
The $\c{D}_{\s P^N,\b Q}^\dagger(\infty_N)$-action on the complex is defined as follows: on each term of the complex, $\c{D}_{\s P^N,\b Q}^\dagger(\infty_N)$ acts by the morphism \[\c{D}_{\s P^N,\b Q}^\dagger(\infty_N)\to \rho'_{2,\ast}\rho'^{-1}_2\c{D}_{\s P^N,\b Q}^\dagger(\infty_N)\to \rho'_{2,\ast}\c{D}_{\s X(\Sigma')\times\s P^N,\b Q}^\dagger(Z').\] One
checks that it commutes with the differential of the complex. We have the following proposition:
\begin{prop}\label{CL}
    There is a canonical isomorphism of $\c{D}_{\s P^N,\b Q}^\dagger(\infty_N)$-complexes:
    \begin{equation}
C_\gamma^{\bullet}(\c L^\dagger)\xrightarrow{\sim} \rho'_{2,\ast}\m{sp}_\ast\Big((j'\times j_N)^\dagger \Omega^\bullet_{X_K(\Sigma')^{\m{an}}\times \b P^{N,{\m{an}}}_K/\b P^{N,{\m{an}}}_K}\big(\m{log} (D'\times \b P^N)\big) \otimes \rho_1'^\ast \widetilde{\c{K}}_\gamma\otimes \widetilde{\c L}_{\pi,\overline F}\Big).
\end{equation}
\end{prop}
\begin{proof}
By \Cref{Ldagger}, the global sections of the right hand side complex is equal to $C_{\gamma}^{\bullet}(L^{\dagger})$. Recall that $L^{\dagger}$ is a coherent $\Gamma(\s P^N, \mathcal D^\dagger_{\mathscr{P}^N, \mathbb Q}(\infty_N))$-module. The identification in global sections gives such a canonical morphism by \cite[Theorem 5.3.3]{Ddaggeraffine}. Each component of the right hand side complex is a finite direct sum of copies of $(j'\times j_N)^\dagger(\c O_{X_K(\Sigma')^{\m{an}}\times \b P^{N,{\m{an}}}_K})$. It is enough to compare the first term on an affine chart $\s U\subset \s P^N$. For $\lambda>1$, the strict neighborhoods $U_\lambda$, $V_\lambda$ are defined in the proof of Lemma \ref{global}.
\begin{eqnarray*}
&& (j'\times j_N)^\dagger
\big( \mathcal O_{X_K(\Sigma')^{\mathrm{an}}\times \mathbb P_K^{N,\mathrm{an}}}\big) \bigl(X_K(\Sigma')^{\mathrm{an}}\times \operatorname{sp}^{-1}(\s U)\bigr)                                                       \\[4pt]
&=&
(j\times j_N)^\dagger \left( \mathcal O_{X_K(\Sigma)^{\mathrm{an}}\times \mathbb P_K^{N,\mathrm{an}}}\right) \bigl(X_K(\Sigma)^{\mathrm{an}}\times \operatorname{sp}^{-1}(\s U)\bigr)
\\[4pt]
&\cong&
\varinjlim_{\lambda\to 1^+} \Gamma\bigl(\operatorname{sp}^{-1}(\s U)\cap U_\lambda, \mathcal O_{\mathbb P_K^{N,\mathrm{an}}}\bigr) \widehat\otimes_K \Gamma\bigl(V_\lambda, \mathcal O_{X_K(\Sigma)^{\mathrm{an}}}\bigr)  \\[4pt]
&=&
\varinjlim_{\lambda\to 1^+} \Gamma\bigl(\operatorname{sp}^{-1}(\s U)\cap U_\lambda, \mathcal O_{\mathbb P_K^{N,\mathrm{an}}}\bigr) \widehat\otimes_{K\langle \lambda^{-1}x_1,\ldots,\lambda^{-1}x_N\rangle} K\langle \lambda^{-1}x_1,\ldots,\lambda^{-1}x_N\rangle \widehat\otimes_K \Gamma\bigl(V_\lambda, \mathcal O_{X_K(\Sigma)^{\mathrm{an}}}\bigr)  \\[4pt]
&=&
\varinjlim_{\lambda\to 1^+} \Gamma\bigl(\operatorname{sp}^{-1}(\s U)\cap U_\lambda, \mathcal O_{\mathbb P_K^{N,\mathrm{an}}}\bigr) \widehat\otimes_{K\langle \lambda^{-1}x_1,\ldots,\lambda^{-1}x_N\rangle} K\{\lambda^{-1}y_1,\ldots,\lambda^{-1}y_M, \lambda^{-1}x_1,\ldots,\lambda^{-1}x_N\}/I \\
&=&\varinjlim_{\lambda\to 1^+}\Gamma\bigl(\operatorname{sp}^{-1}(\s U)\cap U_\lambda, \mathcal O_{\mathbb P_K^{N,\mathrm{an}}}\bigr)\{\lambda^{-1}y_1,\ldots,\lambda^{-1}y_M\}/I\\
&\cong&\big\{\sum_{\bf{w}\in \b Z^n\cap \delta} a_{\bf{w}} \bf{t^w} : a_{\bf{w}}\in \Gamma\bigl(\operatorname{sp}^{-1}(\s U)\cap U_\lambda,
\mathcal O_{\mathbb P_K^{N,\mathrm{an}}}\bigr) \text{ for some } \lambda>1, |a_{\bf{w}}|_\lambda s^{|\bf{w}|} \text{ is bounded for some } s>1\big\}.
\end{eqnarray*}

The last isomorphism is obtained by applying the proof of
\cite[Lemma 5.2]{padictorus} after adjoining the overconvergent coefficients
$\Gamma\bigl(\operatorname{sp}^{-1}(\s U)\cap U_\lambda,
\mathcal O_{\mathbb P_K^{N,\mathrm{an}}}\bigr).$
We denote the ring in the last line by $L^\dagger(\s U)$.
Recall some definitions of the ring of algebraic differential operators and the module in \cite[Section 5.9]{padictorus}.

\begin{eqnarray*}
D
&=&
\left\{ \sum_{\bf{v}\in\mathbb Z_{\geq 0}^{N}} f_\bf{v}(\bf{x})\partial^\bf{v} \;\middle|\; f_\bf{v}(\bf{x})\in K[\bf{x}],\  f_\bf{v}(\bf{x})=0 \text{ for all but finitely many }\bf{v} \right\},\\
L
&=&
\left\{ \sum_{\bf{w}\in\mathbb Z^n\cap\delta} a_\bf{w}(\bf{x})\bf{t}^\bf{w} \;\middle|\; a_\bf{w}(\bf{x})\in K[\bf{x}],\  a_\bf{w}(\bf{x})=0 \text{ for all but finitely many }\bf{w} \right\}.
\end{eqnarray*}
Here
\[ \partial^\bf{v}=\partial_1^{v_1}\cdots\partial_N^{v_N}, \qquad \bf{t}^\bf{w}=t_1^{w_1}\cdots t_n^{w_n}. \]

Set
\[ D^\dagger(\s U) := \Gamma\bigl( \s U,\mathcal D_{\s P^N,\mathbb Q}^{\dagger}(\infty_N) \bigr),
\qquad D^\dagger := \Gamma\bigl( \s P^N, \mathcal D_{\s P^N,\mathbb Q}^{\dagger}(\infty_N) \bigr). \]
By \cite[Section 3.7]{padictorus}, the algebraic presentation of the $D$-module $L$ induces a compatible finite presentation of $L^\dagger$ over $D^\dagger$.
By \cite[Theorem 2.3.4(ii)]{Huyghe2004Trans}, we have 
\[D^\dagger(\s{U})=\left\{ \sum_{\bf{v},\bf{u}\in\mathbb Z_{\geq 0}^{N}} f_{\bf{v},\bf{u}}\bf{x}^{\bf{u}}(\partial/\pi)^\bf{v}
\;\middle|\;
f_{\bf{v},\bf{u}}\in \c{O}_{\s{P}^N,\b{Q}}(\s{U}),\ \exists\, (C>0 ,0<\eta<1) \text{ s.t. }|f_{\bf{v},\bf{u}}|_{\s{U}}\leq C\eta^{|\bf{u}+\bf{v}|}\right\}, \]
where $|\cdot|_{\s{U}}$ is the spectral norm on $\c{O}_{\s{P}^N,\b{Q}}(\s{U})$, which is defined in \cite[0.1.4]{Berrigid}.
Therefore, the same argument induces a compatible finite presentation of $L^\dagger(\s U)$ over $D^\dagger(\s U)$. Thus we have the following commutative diagram:
\[ \xymatrix{ D^{r_1} \ar[r]\ar[d] & D^{r_0} \ar[r]\ar[d] & L \ar[r]\ar[d] & 0
\\ (D^\dagger)^{r_1} \ar[r]\ar[d] & (D^\dagger)^{r_0} \ar[r]\ar[d] & L^\dagger \ar[r]\ar[d] & 0
\\ D^\dagger(\s U)^{r_1} \ar[r] & D^\dagger(\s U)^{r_0} \ar[r] & L^\dagger(\s U) \ar[r] & 0 .} \]
Each horizontal sequence is exact. Hence
\[ L^\dagger(\s U) \simeq D^\dagger(\s U)\otimes_{D^\dagger}L^\dagger = C_\gamma^0(\c L^\dagger)(\s U), \]
by \cite[5.3.3]{Ddaggeraffine}. This proves the proposition. 
\end{proof}

\subsection{Canonical morphism from $C_\gamma^{\bullet}(\c L^\dagger)$ to $\m{Hyp}_\gamma$}\label{canon}
The following lemma realizes $\m{Hyp}_\gamma$ as a relative de Rham complex:

\begin{lem}\label{Hypisodagger}
    \[
    \m{Hyp}_\gamma\cong \rho'_{2,\ast}\m{sp}_\ast\Big((j' j_\delta'\times j_N)^\dagger\Omega^\bullet_{X_K(\Sigma')^{\m{an}}\times \b P^{N,{\m{an}}}_K/\b P^{N,{\m{an}}}_K}\otimes\rho_1'^\ast {\c{K}}_\gamma\otimes \widetilde{\c L}_{\pi,F}\Big),
    \]
    where the tensor product is taken over $(j' j_\delta'\times j_N)^\dagger\c O_{X_K(\Sigma')^{\m{an}}\times \b P_K^{N,\m{an}}}$.
\end{lem}

\begin{proof}
    By definition, 
    \[\m{Hyp}_{\gamma}=R\rho'_{2,*}\bigl(\c{D}_{\s{P}^N\leftarrow\s{X}(\Sigma')\times\s{P}^N,\b{Q}}^{\dagger}(\infty_N,Z')\otimes^{L}_{\c{D}_{\s{X}(\Sigma')\times\s{P}^N,\b{Q}}^{\dagger}(^{\dagger}Z')}
    \text{sp}_+(\rho_1'^\ast \c K_\gamma\otimes F^\ast \c L_\pi)\bigr)[-n].\]
    By \cite[(4.2.1.1)]{BerthelotIntroduction}, we have the following relative Spencer resolution \[\Omega_{\s{X}(\Sigma')\times\s{P}^N/\s{P}^N,\b{Q}}^{\bullet}(^{\dagger}Z')\otimes_{\c{O}_{\s{X}(\Sigma')\times\s{P}^N,\b{Q}}(^{\dagger}Z')}
    \c{D}_{\s{X}(\Sigma')\times\s{P}^N,\b{Q}}^{\dagger}(^{\dagger}Z')[n]\to \c{D}_{\s{P}^N\leftarrow\s{X}(\Sigma')\times\s{P}^N,\b{Q}}^{\dagger}(\infty_N,Z').\]
    Therefore, $\m{Hyp}_{\gamma}$ can be regarded as the complex
    \[R\rho'_{2,*}\m{sp}_\ast\Big((j' j_\delta'\times j_N)^\dagger\Omega^\bullet_{X_K(\Sigma')^{\m{an}}\times \b P^{N,{\m{an}}}_K/\b P^{N,{\m{an}}}_K}\otimes\rho_1'^\ast {\c{K}}_\gamma\otimes \widetilde{\c L}_{\pi,F}\Big).\]
    Each component in the bracket is a finite direct sum of $(j' j_\delta'\times j_N)^\dagger\c O_{X_K(\Sigma')^{\m{an}}\times \b P_K^{N,\m{an}}}$. It suffices to show $R\rho'_{2,*}\m{sp}_\ast\big((j' j_\delta'\times j_N)^\dagger\c O_{X_K(\Sigma')^{\m{an}}\times \b P_K^{N,\m{an}}}\big)\cong \rho'_{2,*}\m{sp}_\ast\big((j' j_\delta'\times j_N)^\dagger\c O_{X_K(\Sigma')^{\m{an}}\times \b P_K^{N,\m{an}}}\big)$.
    
    For an affine formal open subset $\s U\subset \s P^N$ with special fiber $U_\kappa$, $k>0$,  we have
    \begin{align*}
        & \,R^k\rho'_{2,*}\m{sp}_\ast\big((j' j_\delta'\times j_N)^\dagger\c O_{X_K(\Sigma')^{\m{an}}\times \b P_K^{N,\m{an}}}\big)(\s U)
        \\ = &\,  H^k(\rho'^{-1}_{2}(\s U),\m{sp}_\ast((j' j_\delta'\times j_N)^\dagger\c O_{X_K(\Sigma')^{\m{an}}\times \b P_K^{N,\m{an}}})|_{\rho'^{-1}_{2}(\s U)}).
    \end{align*}
For $\rho,\lambda>1$, let
$$T_\rho=\{\bf{t}\in \b T_K^{n,\text{an}} | \rho^{-1}<|t_i|<\rho, \text{ for } i=1,2,\ldots ,n\}\subset \b T_K^{n,\text{an}}\subset X_K(\Sigma')^{\m{an}}.$$
Then $\{T_\rho\times U_\lambda\}_{\rho,\lambda>1}$ forms a cofinal system of strict neighborhoods of $]T_\kappa^n\times \b A_\kappa^N[$ in $X_K(\Sigma')^{\m{an}}\times \b P_K^{N,\m{an}}$.
If $\mathscr V\subset \s X(\Sigma')\times \s P^N$ is an affine formal open subset with special fiber $\s V$, both $\text{sp}^{-1}(V_\kappa)$ and $\text{sp}^{-1}(V_\kappa)\cap (T_\rho\times U_\lambda)$ are affinoid. Hence, by Tate acyclicity theorem, for $k>0$,
$$H^k(\text{sp}^{-1}(V_\kappa),(j' j_\delta'\times j_N)^\dagger\c O_{X_K(\Sigma')^{\m{an}}\times \b P_K^{N,\m{an}}})=\varinjlim_{\rho,\lambda\to1^+}H^k(\text{sp}^{-1}(V_\kappa)\cap (T_\rho\times U_\lambda), O_{X_K(\Sigma')^{\m{an}}\times \b P_K^{N,\m{an}}})=0.$$ Consequently
$$
R\operatorname{sp}_*\big((j' j_\delta'\times j_N)^\dagger\c O_{X_K(\Sigma')^{\m{an}}\times \b P_K^{N,\m{an}}}\big)
\cong\operatorname{sp}_*\big((j' j_\delta'\times j_N)^\dagger\c O_{X_K(\Sigma')^{\m{an}}\times \b P_K^{N,\m{an}}}).
$$
We thus have, for $k>0$,
    \begin{align*}
        & \,H^k(\rho'^{-1}_{2}(\s U),\m{sp}_\ast(j' j_\delta'\times j_N)^\dagger\c O_{X_K(\Sigma')^{\m{an}}\times \b P_K^{N,\m{an}}})
        \\ = &\,  H^k\big(T_K^{n,\text{an}}\times \text{sp}^{-1}(\s U),(j' j_\delta'\times j_N)^\dagger\c O_{X_K(\Sigma')^{\m{an}}\times \b P_K^{N,\m{an}}}\big)
         \\ = &\, \varinjlim_{\rho,\lambda\to1^+}H^k(T_\rho\times(\text{sp}^{-1}(\s U)\cap U_\lambda), \c O_{X_K(\Sigma')^{\m{an}}\times \b P_K^{N,\m{an}}})=0.
    \end{align*}
Since $X_K(\Sigma')^{\mathrm{an}}\times\operatorname{sp}^{-1}(\s U)$ is quasi-compact and separated, and each $T_\rho\times(\text{sp}^{-1}(\s U)\cap U_\lambda)$ is affinoid. This proves our proposition.
\end{proof}

From the sequence of open immersions
$\b T^n\times \b A^N\xrightarrow{\,j'_\delta\times \m{id}\,} X'_\delta\times \b A^N \xrightarrow{\,j'\times j_N\,} X(\Sigma')\times \b P^N,$
we obtain a natural restriction morphism
\[ (j'\times j_N)^\dagger \bigl( \c O_{X_K(\Sigma')^{\m{an}}\times \b P_K^{N,\m{an}}} \bigr) \longrightarrow (j'j'_\delta\times j_N)^\dagger \bigl( \c O_{X_K(\Sigma')^{\m{an}}\times \b P_K^{N,\m{an}}} \bigr).\]
This morphism extends naturally to a morphism of the corresponding twisted de Rham complexes:
\begin{align*}
&\rho'_{2,*}\m{sp}_*
\Bigl( (j'\times j_N)^\dagger \Omega^\bullet_{X_K(\Sigma')^{\m{an}}\times \b P_K^{N,\m{an}}/\b P_K^{N,\m{an}}} \bigl(\log(D'\times \b P^N)\bigr) \otimes \rho_1'{}^\ast \widetilde{\c K}_\gamma \otimes \widetilde{\c L}_{\pi,\overline F} \Bigr)              
\\ & \longrightarrow \rho'_{2,*}\m{sp}_* \Bigl( (j'j'_\delta\times j_N)^\dagger \Omega^\bullet_{X_K(\Sigma')^{\m{an}}\times \b P_K^{N,\m{an}}/\b P_K^{N,\m{an}}} \otimes \rho_1'{}^\ast \c K_\gamma \otimes \widetilde{\c L}_{\pi,F} \Bigr).
\end{align*}
The compatibility with the differentials follows directly from the definitions.
Hence, using the preceding identifications, we obtain a natural morphism
\[ \theta:C_\gamma^\bullet(\c L^\dagger)\longrightarrow \m{Hyp}_\gamma . \]
In the next section, we will give conditions under which $\theta$ is a quasi isomorphism on every affine open $\s U$. We first introduce the following notation to describe the sections of
$C^\bullet_\gamma(\c L^\dagger)$ and $\m{Hyp}_\gamma$.

Recall $\delta\subseteq\b{R}^n$ is the cone generated by $\mathbf{w}_1, \ldots, \mathbf{w}_N$. Let $H_1,\cdots,H_s\in\b{Z}[t_1,\cdots,t_n]$ be primitive integral linear forms defining the codimension one faces of $\delta$, such that for any $\bf{v}\in\b{R}^n$, we have $\bf{v}\in\delta$ iff $H_b(\bf{v})\geqslant 0$ for any $b=1,\cdots,s$. Let $\delta_0=\b{R}^n$, for $b=1,\cdots,s$, let $\delta_b=\{\bf{v}\in\b{R}^n\,|\,H_1(\bf{v}),\cdots,H_b(\bf{v})\geqslant 0\}$. We have $\delta_s=\delta$.

Fix an affine open subset $\s U\subseteq \s P^N$. For $0\leq b\leq s,\lambda,\rho>1$, let
\[L_{b,\lambda,\rho}(\s U)=\big\{\sum_{\mathbf w\in\delta_b\cap \b Z^n}a_{\mathbf w}\bf{t}^{\mathbf w} \big| a_{\mathbf w}\in \Gamma\big(\operatorname{sp}^{-1}(\s U)\cap U_\lambda,\mathcal O_{\mathbb P_K^{N,\mathrm{an}}} \big),\ |a_{\mathbf w}|_{\lambda}\rho^{|\mathbf w|}\to 0  \text{ as } |\bf{w}|\to \infty \big\},\]
and let \[ L_b^\dagger(\s U)=\bigcup_{\lambda>1,\rho>1}L_{b,\lambda,\rho}(\s U).\]
where $|\cdot|_{\lambda}$ is the spectral norm on $\Gamma\big(\operatorname{sp}^{-1}(\s U)\cap U_\lambda,\mathcal O_{\mathbb P_K^{N,\mathrm{an}}} \big)$, which is defined in \cite[0.1.4]{Berrigid}.
Define $$E_{i,\gamma}=t_i\partial_{t_i}+\sum_{j=1}^N\pi w_{ij}x_j\mathbf{t}^{\mathbf{w}_j}+\gamma_i.$$
We will explain how $E_{i,\gamma}$ act on $L_b^\dagger(\s U)$. Recall that $x_1,\cdots, x_N$ are coordinates of $\s A^N\subset \s P^N$. They can be regarded as sections of $\c O_{\b P_K^{N,\text{an}}}$ over $U_\lambda$, and their restrictions to $\operatorname{sp}^{-1}(\s U)\cap U_\lambda$ give elements of $L_b^\dagger(\s U)$. Hence left multiplication by $x_j$ is well defined. Left multiplication by $\bf{t}^{\mathbf w_j}$ is well defined since $\mathbf w_j\in \delta_b$.
For $k=0,\cdots,n$, let
\[C^k\big(L_b^{\dagger}(\s U)\big)=\bigoplus_{1\leqslant i_1<\cdots<i_k\leqslant n}L_b^{\dagger}(\s U)\frac{dt_{i_1}}{t_{i_1}}\wedge\cdots\wedge\frac{dt_{i_k}}{t_{i_k}}.\]
 We can define the complex $C_{\gamma}^{\bullet}\big(L_b^{\dagger}(\s U)\big)$ via the differential maps \[d_k:C^k(L_b^{\dagger}(\s U))\to C^{k+1}(L_b^{\dagger}(\s U)),\quad\xi\frac{dt_{i_1}}{t_{i_1}}\wedge\cdots\wedge\frac{dt_{i_k}}{t_{i_k}}\mapsto\sum_{i=1}^nE_{i,\gamma}(\xi)\frac{dt_i}{t_i}\wedge\frac{dt_{i_1}}{t_{i_1}}\wedge\cdots\wedge\frac{dt_{i_k}}{t_{i_k}}.\]
Obviously, by \Cref{CL} and \Cref{Hypisodagger}, we have
\[C^\bullet_\gamma(\c L^\dagger)(\s U)=C_{\gamma}^{\bullet}\big(L_s^{\dagger}(\s U)\big)\qquad
\operatorname{Hyp}_{\gamma}(\s U)=C_{\gamma}^{\bullet}\big(L_0^{\dagger}(\s U)\big).\]
Thus, we obtain a sequence of inclusions of complexes
\[ C_{\gamma}^{\bullet}\bigl(L^\dagger(\s U)\bigr) =
C_{\gamma}^{\bullet}\bigl(L_s^\dagger(\s U)\bigr) \hookrightarrow
C_{\gamma}^{\bullet}\bigl(L_{s-1}^\dagger(\s U)\bigr) \hookrightarrow \cdots \hookrightarrow
C_{\gamma}^{\bullet}\bigl(L_1^\dagger(\s U)\bigr) \hookrightarrow
C_{\gamma}^{\bullet}\bigl(L_0^\dagger(\s U)\bigr). \]
In the next section, we give conditions under which the natural morphism
\[ C_{\gamma}^{\bullet}\bigl(L_b^\dagger(\s U)\bigr) \longrightarrow
C_{\gamma}^{\bullet}\bigl(L_{b-1}^\dagger(\s U)\bigr) \]
is a quasi-isomorphism for each $b=1,\ldots,s$. By applying these quasi-isomorphisms successively, we will prove Theorem \ref{main}.

\section{Proof of Comparison theorem}\label{Comparison}

\subsection{$p$-adic Liouville numbers}\label{Liouville}
We recall some basic facts about the $p$-adic Liouville numbers from \cite{defnofLiouville}.

We say that an element $\lambda\in K$ is Liouville if there exist infinitely many integers $m$ and a real number $0<\eta<1$, such that $|\lambda-m|<\eta^{|m|}$. Equivalently, if $\lambda\in K$ is non-Liouville, then for any $0<\eta<1$, there exists $c_{\eta}>0$ such that \[\left|\frac{1}{\lambda-m}\right|\leq c_{\eta}\eta^{-|m|}\] for any $m\neq \lambda$. The numbers in $\b{Q}^{\m{alg}}\cap K$ are non-Liouville, where $\b{Q}^{\m{alg}}$ is an algebraic closure of the rational number field $\b{Q}$.

By definition, for any $m\in\b{Z}$ and $\lambda\in K$, $\lambda$ is a $p$-adic Liouville number if and only if $\lambda+m$ is a $p$-adic Liouville number. We can construct a $p$-adic Liouville number as follows: set $r_1=1$ and define inductively $r_{l+1}=(l+1)p^{r_l}$ for $l\geq 1$. Let $\lambda:=\sum_{l=1}^{\infty}p^{r_l}$, and let $[\lambda]_m:=\sum_{l=1}^{m}p^{r_l}$. Then we have $|\lambda-[\lambda]_m|=p^{-r_{m+1}}<p^{-[\lambda]_m}$. Hence $\lambda$ is a $p$-adic Liouville number.

\subsection{$p$-adic semi-nonresonance condition}\label{nonresonance}

We say $\gamma=(\gamma_1,\cdots,\gamma_n)\in K^n$ is \textit{$p$-adically semi-nonresonant} if $H_b(\gamma)$ is non-Liouville and does not belong to $\b{Z}_{>0}$ for any $b=1,\cdots,s$.

\begin{prop}\label{comp}
    Fix $b$. If $H_b(\gamma)$ is non-Liouville and does not belong to $\b{Z}_{>0}$, then the canonical morphism $C_{\gamma}^{\bullet}\big(L_b^{\dagger}(\s U)\big)\to C_{\gamma}^{\bullet}\big(L_{b-1}^{\dagger}(\s U)\big)$ is a quasi-isomorphism.
\end{prop}

\subsection{Facetwise exponential conjugation}

Fix $1\leq b\leq s$. Let \[Q_b^\dagger:=L_{b-1}^{\dagger}(\s U)/L_b^{\dagger}(\s U)\cong\big\{\sum_{\bf{w}\in \delta_{b-1}\cap\b{Z}^n}a_{\bf{w}}\mathbf{t}^{\bf{w}}\in L_{b-1}^{\dagger}(\s U)\,|\,H_b(\bf{w})\leq -1\big\}.\]
We separate $F(\bf{x},\bf{t})=\sum_{j=1}^N x_j \bf{t}^{\bf{w}_j}=F_b(\bf{x},\bf{t})+F_{b,+}(\bf{x},\bf{t}),$ where
\[ F_b(\bf{x},\bf{t})=\sum_{H_b(\bf{w}_j)=0}x_j \bf{t}^{\bf{w}_j},\quad F_{b,+}(\bf{x},\bf{t})=\sum_{H_b(\bf{w}_j)>0}x_j \bf{t}^{\bf{w}_j}.\]
Let $\m{pr}_b: A[[\bf{t}^{\bf{w}}: \bf{w}\in \delta_{b-1}\cap \b Z^n]]\to A[[\bf{t}^{\bf{w}}: \bf{w}\in (\delta_{b-1}\setminus\delta_{b})\cap\b Z^n]] $ be the projection:
\[ \m{pr}_b(\sum_{\bf{w}\in \delta_{b-1}\cap \b Z^n} a_{\bf{w}}\bf{t}^{\bf{w}})=\sum_{\bf{w}\in (\delta_{b-1}\setminus\delta_{b})\cap\b Z^n} a_{\bf{w}}\bf{t}^{\bf{w}},\]
where $A$ is a ring. We have
\begin{lem}\label{inverse}
The operator $\phi_b:=\m{pr}_b\circ \exp (\pi F_{b,+}(\bf{x},\bf{t}))$ is well defined on $Q_b^\dagger$, and it has an inverse map $\m{pr}_b\circ \exp (-\pi F_{b,+}(\bf{x},\bf{t}))$.
\end{lem}
\begin{proof}
For $\lambda,\rho>1$, let $$Q_{b,\lambda,\rho}:=L_{b-1,\lambda,\rho}(\s U)/L_{b,\lambda,\rho}(\s U)=\big\{\sum_{\bf{w}\in \delta_{b-1}\cap\b{Z}^n,H_b(\bf{w})\le -1}a_{\bf{w}}\mathbf{t}^{\bf{w}}\big|\,|a_{\bf{w}}|_\lambda \rho^{|\bf{w}|}\to 0 \text{ as } |\bf{w}|\to \infty\big\}.$$
We have
\[Q_b^\dagger=\bigcup_{\lambda>1,\rho>1}Q_{b,\lambda,\rho}.\]
Let $J_b=\{j| 1\le j\le N, H_b(\bf{w}_j)>0\}$. For $a_{\bf{w}}\mathbf{t}^{\bf{w}}\in Q_{b,\lambda,\rho}$, we have
\[\big(\m{pr}_b\circ \exp(\pi F_{b,+}(\bf{x},\bf{t}))\big)(a_{\bf{w}}\mathbf{t}^{\bf{w}})=\sum_{\substack{\bf{k}=(k_j)_{j\in J_b}\in \b Z_{\geq 0}^{|J_b|}
\\ H_b(\bf{w})+\sum_{j\in J_b}k_jH_b(\bf{w}_j)\le -1}}\frac{\pi^{|\bf{k}|}}{\prod_{j\in J_b}k_j!}a_{\bf{w}}\bf{x}^{\bf{k}} \bf{t}^{\bf{w}+\sum_{j\in J_b}k_j\bf{w}_j}.\]
The condition $H_b(\bf{w})+\sum_{j\in J_b}k_jH_b(\bf{w}_j)\le -1$ implies $|\bf{k}|\le \frac{-H_b(\bf{w})}{\min_{j\in J_b}H_b(\bf{w}_j)}$ and hence the above sum is finite. Moreover, there exist constants $C_1$ and $C_2$ such that
$$|\bf{k}|\le  \frac{-H_b(\bf{w})}{\min_{j\in J_b}H_b(\bf{w}_j)}\le C_1\cdot|\bf{w}|, \quad |\bf{w}+\sum_{j\in J_b}k_j\bf{w}_j|\le |\bf{w}|+|\bf{k}|\cdot \max_{j\in J_b}|\bf{w}_j|\le C_2\cdot|\bf{w}|.$$
Here $C_1$ depends on the coefficients of the linear form $H_b$ and the values $H_b(\bf{w}_j)$, while $C_2$ depends on $C_1$ and $|\bf{w}_j|$. Let $\lambda'=\rho'>1$ such that $\rho'^{C_1+C_2}<\rho$ and $1<\lambda'<\lambda$. Viewing the above sum as an element in $Q_{b, \lambda',\rho'}$, we obtain
\begin{eqnarray*}
&&\big|\frac{\pi^{|\bf{k}|}}{\prod_{j\in J_b}k_j!}\big|\cdot |a_{\bf{w}}|_{\lambda'}|\bf{x}^{\bf{k}}|_{\lambda'}\cdot {\rho'}^{|\bf{w}+\sum_{j\in J_b}k_j\bf{w}_j|}\\
&\le&|a_{\bf{w}}|_{\lambda}{\lambda'}^{|\bf{k}|}{\rho'}^{|\bf{w}+\sum_{j\in J_b}k_j\bf{w}_j|}\\
&\le&|a_{\bf{w}}|_{\lambda}{\rho'}^{(C_1+C_2)|\bf{w}|}\\
&\le&|a_{\bf{w}}|_{\lambda}\rho^{|\bf{w}|}.
\end{eqnarray*}
The above norm estimate shows that $\m{pr}_b\circ \exp (\pi F_{b,+}(\bf{x},\bf{t}))$ is continuous on the subspace of finite sums in $Q_{b,\lambda,\rho}$. Since this subspace is dense and $Q_{b,\lambda',\rho'}$ is complete, it extends uniquely to a continuous map from $Q_{b,\lambda,\rho}$ to $Q_{b,\lambda',\rho'}$. Passing to the limit proves the first assertion.

Let $\xi\in Q_b^\dagger$ be represented by a power series supported in $\delta_{b-1}\setminus\delta_b$, and let
$$\exp (\pi F_{b,+}(\bf{x},\bf{t}))\cdot \xi=\xi_1+\xi_2,$$
where $\xi_1$ and $\xi_2$ are supported in $\delta_{b-1}\setminus\delta_b$ and $\delta_b$, respectively. We have
\begin{eqnarray*}
&&\big(\m{pr}_b\circ \exp (-\pi F_{b,+}(\bf{x},\bf{t}))\big)\circ \big(\m{pr}_b\circ \exp (\pi F_{b,+}(\bf{x},\bf{t}))\big)(\xi)\\
&=&\big(\m{pr}_b\circ \exp (-\pi F_{b,+}(\bf{x},\bf{t}))\big)(\exp (\pi F_{b,+}(\bf{x},\bf{t}))\cdot \xi-\xi_2)\\
&=&\m{pr}_b\big(\xi-\exp (-\pi F_{b,+}(\bf{x},\bf{t}))\cdot\xi_2\big)\\
&=&\xi.
\end{eqnarray*}
The last equality follows from the fact that $\exp (-\pi F_{b,+}(\bf{x},\bf{t}))$ and $\xi_2$ are supported in $\delta_b$ and thus annihilated by $\m{pr}_b$. Interchanging $F_{b,+}(\bf{x},\bf{t})$ and $-F_{b,+}(\bf{x},\bf{t})$ gives the reverse composition. Thus the two operators are mutually inverse.
\end{proof}

The quotient complex $C_\gamma^\bullet(Q_b^\dagger)$ of $C_{\gamma}^{\bullet}\bigl(L_b^\dagger(\s U)\bigr)\rightarrow C_{\gamma}^{\bullet}\bigl(L_{b-1}^\dagger(\s U)\bigr)$ is given by
$$C^k_\gamma(Q_b^\dagger)=\bigoplus_{1\leq i_1<\cdots<i_k\leq n}Q_b^\dagger\frac{dt_{i_1}}{t_{i_1}}\wedge\cdots\wedge\frac{dt_{i_k}}{t_{i_k}}.$$
The differential map $d_k:C^k_\gamma(Q_b^\dagger)\to C^{k+1}_\gamma(Q_b^\dagger)$ is given by
\[d_k:C^k_{\gamma}(Q_b^{\dagger})\to C_{\gamma}^{k+1}(Q_b^{\dagger}),\quad\xi\frac{dt_{i_1}}{t_{i_1}}\wedge\cdots\wedge\frac{dt_{i_k}}{t_{i_k}}\mapsto\sum_{i=1}^n\m{pr}_b\circ E_{i,\gamma}(\xi)\frac{dt_i}{t_i}\wedge\frac{dt_{i_1}}{t_{i_1}}\wedge\cdots\wedge\frac{dt_{i_k}}{t_{i_k}},\]
where $E_{i,\gamma}=t_i\partial_{t_i}+\sum_{j=1}^N\pi w_{ij}x_j\mathbf{t}^{\mathbf{w}_j}+\gamma_i.$
We define a twisted complex $\tilde{C}_\gamma^\bullet(Q_b^\dagger)$ to have the same underlying graded module as $C_\gamma^\bullet(Q_b^\dagger)$, with differential $\tilde d_k$ obtained by replacing $E_{i,\gamma}$ by
 $$\tilde{E}_{i,\gamma}=t_i\partial_{t_i}+\sum_{j\notin J_b}\pi w_{ij}x_j\mathbf{t}^{\mathbf{w}_j}+\gamma_i.$$
\begin{prop}\label{trans}
Applying the operator $\phi_b=\m{pr}_b\circ \exp (\pi F_{b,+}(\bf{x},\bf{t}))$ to each factor of every degree defines an isomorphism of complexes
$$C_\gamma^\bullet(Q_b^\dagger)\to \tilde{C}_\gamma^\bullet(Q_b^\dagger).$$
\end{prop}

\begin{proof}
By Lemma \ref{inverse}, it suffices to verify, for each $i$, the commutativity of the following diagram:
\[\begin{tikzcd}[column sep=large]
Q_b^\dagger \arrow[r,"\m{pr}_b\circ E_{i,\gamma}"] \arrow[d,"\phi_b"'] & Q_b^\dagger \arrow[d,"\phi_b"]
\\ Q_b^\dagger \arrow[r,"\m{pr}_b\circ \tilde{E}_{i,\gamma}"'] & Q_b^\dagger
\end{tikzcd}\]
By definition, we have
\begin{align*}
    E_{i,\gamma} & =\exp(-\pi F(\bf{x},\bf{t}))\circ t_i\partial_{t_i}\circ \exp(\pi F(\bf{x},\bf{t}))+\gamma_i, \\ \tilde{E}_{i,\gamma} & =\exp(-\pi F_b(\bf{x},\bf{t}))\circ t_i\partial_{t_i}\circ \exp(\pi F_b(\bf{x},\bf{t}))+\gamma_i.
\end{align*}
Hence
\[\widetilde E_{i,\gamma}\circ\exp(\pi F_{b,+}(\bf{x},\bf{t}))=\exp(\pi F_{b,+}(\bf{x},\bf{t}))\circ E_{i,\gamma}.\]
Since the operators
$\exp(\pi F_{b,+})$ and $\widetilde E_{i,\gamma}$ preserve the subspace of power series supported in $\delta_b$, we have
\[ \operatorname{pr}_b\circ\exp(\pi F_{b,+})\circ\operatorname{pr}_b = \operatorname{pr}_b\circ\exp(\pi F_{b,+}),\quad 
\operatorname{pr}_b\circ\widetilde  E_{i,\gamma}\circ\operatorname{pr}_b = \operatorname{pr}_b\circ\widetilde E_{i,\gamma}. \]
For $\xi\in Q_b^\dagger$, we have
\begin{eqnarray*}
\phi_b\circ \m{pr}_b\circ E_{i,\gamma}(\xi)&=&\m{pr}_b\circ \exp(\pi F_{b,+}(\bf{x},\bf{t}))\circ\m{pr}_b\circ E_{i,\gamma}(\xi)\\
&=&\m{pr}_b\circ \exp(\pi F_{b,+}(\bf{x},\bf{t}))\circ E_{i,\gamma}(\xi)\\
&=&\m{pr}_b\circ\tilde E_{i,\gamma}\circ\exp(\pi F_{b,+}(\bf{x},\bf{t}))(\xi)\\
&=&\operatorname{pr}_b\circ\tilde E_{i,\gamma}\circ\operatorname{pr}_b\circ\exp(\pi F_{b,+}(\bf{x},\bf{t}))(\xi)\\
&=&\operatorname{pr}_b\circ\tilde E_{i,\gamma}\circ \phi_b(\xi).
\end{eqnarray*}
This proves the proposition.
\end{proof}

\medskip

\subsection{Proof of \Cref{comp}}
Since the canonical morphism $L_b^{\dagger}(\s U)\to L_{b-1}^{\dagger}(\s U)$ is injective, we have \[ \m{Cone}\Big(C_{\gamma}^{\bullet}\big(L_b^{\dagger}(\s U)\big)\to C_{\gamma}^{\bullet}\big(L_{b-1}^{\dagger}(\s U)\big)\Big)\cong C_\gamma^\bullet(Q_b^\dagger)\] in the derived category. Therefore, we can prove the assertion by showing $C_\gamma^\bullet(Q_b^\dagger)$ is acyclic. By Proposition \ref{trans}, it suffices to prove the acyclicity of  $\tilde{C}_\gamma^\bullet(Q_b^\dagger)$. Let $H_b(z)=\sum_{i=1}^n c_i z_i\in \b Z[z_1,\cdots, z_n]$ be a defining linear form for the hyperplane. Define
$$\iota_b:C^k(Q_b^\dagger)\to C^{k-1}(Q_b^\dagger): \xi\frac{dt_{i_1}}{t_{i_1}}\wedge\cdots\wedge\frac{dt_{i_k}}{t_{i_k}}\mapsto\sum_{j=1}^k(-1)^{j-1}c_{i_j}\xi\frac{dt_{i_1}}{t_{i_1}}\wedge\cdots\wedge\widehat{\frac{dt_{i_j}}{t_{i_j}}}\wedge\cdots\wedge\frac{dt_{i_k}}{t_{i_k}}.$$
A direct computation gives
\begin{eqnarray*}
(\iota_b\circ \tilde d_k+\tilde d_{k-1}\circ \iota_b)\big(\xi\frac{dt_{i_1}}{t_{i_1}}\wedge\cdots\wedge\frac{dt_{i_k}}{t_{i_k}}\big)&=&\sum_{i=1}^n \m{pr}_b\circ c_i \tilde E_{i,\gamma}(\xi)\frac{dt_{i_1}}{t_{i_1}}\wedge\cdots\wedge\frac{dt_{i_k}}{t_{i_k}}\\
&=&\big(\sum_{i=1}^n c_it_i\partial_{t_i}+H_b(\gamma)\big)(\xi)\frac{dt_{i_1}}{t_{i_1}}\wedge\cdots\wedge\frac{dt_{i_k}}{t_{i_k}}.
\end{eqnarray*}
Thus the endomorphism $\alpha^\bullet$ of $\tilde{C}_\gamma^\bullet(Q_b^\dagger)$ defined by
$$\alpha_k\big(\xi\frac{dt_{i_1}}{t_{i_1}}\wedge\cdots\wedge\frac{dt_{i_k}}{t_{i_k}}\big)= \big(\sum_{i=1}^nc_it_i\partial_{t_i}+H_b(\gamma)\big)(\xi)\frac{dt_{i_1}}{t_{i_1}}\wedge\cdots\wedge\frac{dt_{i_k}}{t_{i_k}}$$
is null-homotopic. It therefore remains to show that $\alpha^\bullet$ is an automorphism. For $\xi=\sum_{\bf{w}\in \delta_{b-1}\cap\b Z^n, H_b(\bf{w})\le -1}a_{\bf{w}}t^{\bf{w}}\in Q_b^\dagger$, define
$$\beta_0(\xi)=\sum_{\bf{w}\in\delta_{b-1}\cap\b{Z}^n,H_b(\bf{w})\leq -1} \frac{a_{\bf{w}}}{H_b(\bf{w})+H_b(\gamma)}\mathbf{t}^{\bf{w}}.$$
It is well defined since $H_b(\gamma)\notin \b Z_{>0}$.
Moreover, $\beta_0\alpha_0=\alpha_0\beta_0=\text{Id}_{Q_b^\dagger}$. Now we prove $\beta_0(\xi)\in Q_b^\dagger$.
Suppose $\xi\in Q_{b,\lambda,\rho}$ such that
$$|a_{\bf{w}}|_\lambda \rho^{|\bf{w}|}\to 0 \text{ as }|\bf{w}|\to \infty.$$
Since $H_b(\gamma)$ is $p$-adically non-Liouville, for any $\rho'>1$, there exists a constant $C$ such that
$$\big| \frac{1}{H_b(\bf{w})+H_b(\gamma)}\big|\le C{\rho'}^{|H_b(\bf{w})|}.$$
Choose $\rho'>1$ such that ${\rho'}^{1+\max_i |c_i|}\le \rho$. We obtain
\begin{eqnarray*}
\big|\frac{a_{\bf{w}}}{H_b(\bf{w})+H_b(\gamma)}\big|_\lambda \rho'^{|\bf{w}|}&\le& |a_{\bf{w}}|_\lambda \cdot \big| \frac{1}{H_b(\bf{w})+H_b(\gamma)}\big| \rho'^{|\bf{w}|}\\
&\le&|a_{\bf{w}}|_\lambda \cdot C{\rho'}^{|H_b(\bf{w})|+|\bf{w}|}\\
&\le& C\cdot |a_{\bf{w}}|_\lambda \rho^{|\bf{w}|}.
\end{eqnarray*}
It tends to $0$ when $|\bf{w}|\to \infty$. Hence $\beta_0(Q_{b,\lambda,\rho})\subset Q_{b,\lambda,\rho'}.$ Applying $\beta_0$ termwise in each degree yields an inverse $\beta^\bullet$ of $\alpha^\bullet$ on $\tilde{C}_\gamma^\bullet(Q_b^\dagger)$. This proves the proposition.

\begin{proof}[Proof of \Cref{main}]
    By applying \Cref{comp} successively, this follows from the facts $C^\bullet_\gamma(\c L^\dagger)=C_{\gamma}^{\bullet}\big(\c{L}_s^{\dagger}\big)$
    and $\operatorname{Hyp}_{\gamma}=C_{\gamma}^{\bullet}\big(\c{L}_0^{\dagger}\big).$
\end{proof}

Note that $\m{Hyp}_{\gamma}\cong\m{Hyp}_{\gamma-\bf{v}}$ for any vector $\bf{v}\in\b{Z}^n$. We have the following corollary:
\begin{cor}\label{Hypcoherent}
    If the numbers $H_1(\gamma),\cdots,H_s(\gamma)$ are all $p$-adically non-Liouville, then $\m{Hyp}_{\gamma}$ is a coherent $\c{D}_{\s{P}^N,\b{Q}}^{\dagger}(\infty_N)$-complex.
\end{cor}

\begin{rmk}
    $\m{Hyp}_{\gamma}$ is not always coherent as a $\c{D}_{\s{P}^N,\b{Q}}^{\dagger}(\infty_N)$-complex. For example, let $F:\mathbb{T}_{\kappa}^1\times \b A^1_\kappa\to \b{A}_{\kappa}^1, \ \big((t,x)\mapsto xt\big)$ be the torus action on $\b A_{\kappa}^1$, $\gamma$ the $p$-adic Liouville number given in \Cref{Liouville}. We show that $\m{Hyp}_{\gamma}$ is not a coherent $\c{D}_{\s{P}^1,\b{Q}}^{\dagger}(\infty_1)$-complex as follows.

    Consider the global sections. It suffices to show that \[H^1:=\m{Coker}(K\langle x,t^{\pm }\rangle^{\dagger}\xrightarrow{t\partial_t+\gamma+\pi xt}K\langle x,t^{\pm }\rangle^{\dagger})\] is not coherent as a $K\langle x,\frac{\partial_x}{\pi}\rangle^{\dagger}$-module. One checks that the cokernel is isomorphic to $K\langle t^{\pm }\rangle^{\dagger}$ as $K$-linear spaces. For any monomial $t^v$, the $K\langle x,\frac{\partial_x}{\pi}\rangle^{\dagger}$-module structure on $K\langle t^{\pm }\rangle^{\dagger}$ is defined by $x\cdot t^v=-\frac{v+\gamma-1}{\pi}t^{v-1}$ and $\partial_x\cdot t^v=\pi t^{v+1}$.
    
    Assume $H^1$ can be generated by finitely many elements \[g_1=\sum_{u\in\b{Z}}b_{1,u}t^{u},\cdots,g_s=\sum_{u\in\b{Z}}b_{s,u}t^{u}.\] We may choose $c>0,0<\eta<1$ such that $|b_{r,u}|\leq c\eta^{|u|}$ for any $u\in\b{Z},r=1,\cdots,s$.
    Moreover, let $[\gamma]_m:=\sum_{l=1}^{m}p^{r_l}$. After enlarging $\eta$ if necessary, we may assume $|\gamma-[\gamma]_m|<\eta^{[\gamma]_m}$ for all $m$.
    Choose $k$ large enough such that $\eta<p^{-1/k}$. After replacing $K$ by a finite extension, we may assume $p^{1/k}\in K$. Then $f:=\sum_{m=1}^{\infty}p^{\frac{[\gamma]_m}{k}}t^{-[\gamma]_m}$ belongs to $H^1$.
    If there exist $P_1=\sum_{v,w\in\b{Z}_{\geq 0}}a_{1,v,w}x^{v}\frac{\partial_x^{w}}{\pi^w},\cdots,P_s=\sum_{v,w\in\b{Z}_{\geq 0}}a_{s,v,w}x^{v}\frac{\partial_x^{w}}{\pi^w}$ such that $P_1(g_1)+\cdots+P_s(g_s)=f$, then we have 
    \begin{equation}\label{coe}\sum_{1\leq r\leq s,w+u-v=-[\gamma]_m}a_{r,v,w}x^{v}\frac{\partial_x^{w}}{\pi^w}(b_{r,u}t^{u})=p^{\frac{[\gamma]_m}{k}}t^{-[\gamma]_m}.\end{equation}
    The coefficients on the left-hand side are of the following two forms.  For $v\geq 1$, the coefficient is  \[a_{r,v,w}b_{r,u}\frac{(\gamma+u+w-1)\cdots(\gamma+u+w-v)}{(-\pi)^v}=a_{r,v,w}b_{r,u}\frac{(v-1)!}{(-\pi)^v}\binom{\gamma-[\gamma]_m+v-1}{v-1}(\gamma-[\gamma]_m). \]
    For $v=0$ the corresponding coefficient is $a_{r,v,w}b_{r,u}$.
    
    Since $P_1,\cdots, P_s$ belong to $K\langle x,\frac{\partial_x}{\pi}\rangle^{\dagger}$, there exist $c'>0,0<\eta'<1$ such that $|a_{r,v,w}|\leq c'\eta'^{v+w}$ for all $r,v,w$.
    For $v\geq 1$, one checks that the family $\{a_{r,v,w}\frac{(v-1)!}{\pi^v}\}_{r,v,w}$ is bounded. As $\binom{\gamma-[\gamma]_m+v-1}{v-1}\in\b{Z}_p$ and $\{b_{r,u}\}_{r,u}$ are bounded, we have the estimate:
    \[\big|a_{r,v,w}b_{r,u}\frac{(v-1)!}{\pi^v}\binom{\gamma-[\gamma]_m+v-1}{v-1}(\gamma-[\gamma]_m)\big|\le C\eta^{[\gamma]_m}\] for some constant $C$.
    In the case $v=0$, we have
    \[|a_{r,0,w}b_{r,u}|\le cc'\eta'^{w}\eta^{|u|}=cc'\eta'^{w}\eta^{w+[\gamma]_m}\le cc'\eta^{[\gamma]_m}.\]
    Comparing the coefficients in \ref{coe}, we obtain
    \[|p^{\frac{[\gamma]_m}{k}}|=p^{-\frac{[\gamma]_m}{k}}\leq C_0\eta^{[\gamma]_m}\] for some constant $C_0>0$. This leads to a contradiction for sufficiently large $m$ since $\eta<p^{-1/k}$.
\end{rmk}

\section{Proof of \Cref{complex} and \Cref{Dmodule}}\label{moreon}

\begin{prop}\label{Hk=0}
For any $\gamma\in \b{Z}_p^n$, we have $H^k( C^\bullet_\gamma(\c L^\dagger))=0$ for $k<n.$
\end{prop}
\begin{proof}
By $\c D^\dagger(\infty_N)$-affinity, it suffices to consider the complex of global sections $C^\bullet_\gamma(L^\dagger)$.
Recall the following algebraic counterparts defined in the proof of \Cref{CL}:
$$D=K[\bf{x}]\langle\partial_1,\ldots, \partial_N\rangle, \ \ 
  S=K[\bf{t}^{\bf{v}}: \bf{v}\in \delta\cap\b Z^n],\ \  L=S[x_1,\ldots, x_N].
$$
Replacing $L^\dagger$ by $L$ in $C_\gamma^\bullet(L^\dagger)$, we obtain the algebraic logarithmic de Rham complex $C_\gamma^\bullet(L)$. By \cite[Section 5.9]{padictorus}, we have
$$D^\dagger\otimes_D C_\gamma^\bullet(L)\cong C_\gamma^\bullet(L^\dagger),$$
where $D^\dagger=\Gamma(\s P^N, \c D_{\s P^N,\b Q}^\dagger(\infty_N))$. Denote by $K_\bullet(L;\{E_{i,\gamma}\}_{i=1}^n)$ the homological Koszul complex of the commuting operators $E_{1,\gamma},\ldots,E_{n,\gamma}$ on $L$. We have 
$$H^k\big(C_\gamma^\bullet(L)\big)=H_{n-k}\big(L; \{E_{i,\gamma}\}_{i=1}^n\big). $$ Thus it suffices to prove that $D^\dagger$ is flat as a right $D$-module and that 
$$H_{k}\big(L; \{E_{i,\gamma}\}_{i=1}^n\big)=0 \text{ for }k>0.$$

First, let 
$$\widetilde A=\begin{pmatrix} 1& 1 & \ldots & 1 \\ 0 & \bf{w}_1&\ldots& \bf{w}_N \end{pmatrix}.$$
This matrix is pointed and has rank $n+1$.
Denote its column vectors by $\widetilde{\bf{w}}_0, \widetilde{\bf{w}}_1,\ldots, \widetilde{\bf{w}}_N\in \b Z^{n+1}$, and let $\widetilde\delta$ be the cone in $\b R^{n+1}$ generated by those columns. Put
$$\widetilde S=K[z^{v_0}\bf{t}^{\bf{v}}: (v_0,\bf{v})\in \widetilde\delta\cap \b Z^{n+1}], \ \widetilde L=\widetilde S[x_0,x_1,\ldots, x_N],\ \widetilde F=x_0z+\sum_{j=1}^Nx_jz\bf{t}^{\bf{w}_j}\in \widetilde L.$$
Here we regard $z$ as an additional variable adjoined to $\mathbf{t}$. For $\widetilde\gamma=(\gamma_0,\gamma_1,\ldots,\gamma_n)\in K^{n+1}$, define
\begin{eqnarray*}
\widetilde E_{i,\widetilde\gamma}&=&t_i\partial_{t_i}+\gamma_i+\pi\sum_{j=1}^Nw_{ij}x_j z \bf{t}^{\bf{w_j}}, \text{ for } i=1,2,\ldots, n,\\
\widetilde E_{z,\widetilde\gamma}&=&z\partial_z+\gamma_0+\pi \widetilde F.
\end{eqnarray*}
Denote their homological Koszul complex on $\widetilde L$ by $K_\bullet\big(\widetilde L; \{\widetilde E_{z,\widetilde\gamma}, \widetilde E_{1,\widetilde\gamma},\ldots, \widetilde E_{n,\widetilde\gamma}\}\big)$. 
The normal affine semigroup ring $\widetilde S$ is Cohen–Macaulay of dimension $n+1$ by \cite{Hochster72}. Moreover, it is finite over 
$$S'=K[\b Z_{\ge 0}\widetilde{\bf{w}}_0+\cdots+\b Z_{\ge 0}\widetilde{\bf{w}}_N].$$
Hence it is a toric module with respect to $\widetilde A$ by \cite[Definition 4.5 and Example 4.7]{toricmodule}. Therefore \cite[Remark 6.4]{toricmodule} gives
$$H_k\big(\widetilde L; \{\widetilde E_{z,\widetilde\gamma}, \widetilde E_{1,\widetilde\gamma},\ldots, \widetilde E_{n,\widetilde\gamma}\}\big)=0\text{ for }k>0.$$
The cited result over $\mathbb C$ applies here by descending to
$\b Q(\pi,\gamma_0,\ldots,\gamma_n)$, which can be embedded in $\b C$.

\textbf{Claim 1.} We have $\widetilde S[z^{-1}]=S[z^{\pm1}]$.

The inclusion $\widetilde S[z^{-1}]\subset S[z^{\pm1}]$ follows by projecting $\widetilde\delta$ onto $\delta$. Conversely, let $\mathbf v\in\delta\cap\mathbb Z^n$. Write
\[ \mathbf v=\sum_{j=1}^N a_j\mathbf w_j, \qquad a_j\geq0,\]
and choose an integer $h\geq\sum_j a_j$. Then
\[ (h,\mathbf v) = \left(h-\sum_{j=1}^N a_j\right)(1,\mathbf0) +\sum_{j=1}^N a_j(1,\mathbf w_j) \in\widetilde\delta\cap\mathbb Z^{n+1}. \]
Thus $z^h\mathbf t^{\mathbf v}\in\widetilde S$, and hence
$\mathbf t^{\mathbf v}\in\widetilde S[z^{-1}]$. This proves the claim.

Put $\widetilde\gamma_0=(0,\gamma)$, and let $ K_\bullet= K_\bullet\bigl( \widetilde L[z^{-1}]; \widetilde E_{z,\widetilde\gamma_0}, \widetilde E_{1,\widetilde\gamma_0},\ldots, \widetilde E_{n,\widetilde\gamma_0} \bigr)$.

\textbf{Claim 2.} We have $H_k(K_\bullet)=0$ for $k>0$.

For each $m\geq0$, the operators defining $K_\bullet$ preserve
$z^{-m}\widetilde L$. Moreover,
\[ \widetilde E_{z,(0,\gamma)}(z^{-m}f) = z^{-m}\widetilde E_{z,(-m,\gamma)}(f),
\qquad \widetilde E_{i,(0,\gamma)}(z^{-m}f) = z^{-m}\widetilde E_{i,(-m,\gamma)}(f). \]
Consequently, multiplication by $z^{-m}$ identifies the subcomplex with coefficients in $z^{-m}\widetilde L$ with the Koszul complex on $\widetilde L$ with parameter $(-m,\gamma)$. Its homology vanishes in positive degrees by the preceding result. Since
\[ \widetilde L[z^{-1}] =\varinjlim_{m\geq0}z^{-m}\widetilde L, \]
the claim follows from the exactness of filtered direct limits.

We now make the change of variables $y_j=zx_j, (0\leq j\leq N)$. 
By Claim 1, \[ \widetilde L[z^{-1}] =S[z^{\pm1},y_0,y_1,\ldots,y_N]. \]
Under this identification, the operators $\widetilde E_{i,(0,\gamma)}$ ($1\leq i\leq n$) become
\[ E_{i,\gamma}(\mathbf y) = t_i\partial_{t_i}+\gamma_i +\pi\sum_{j=1}^N w_{ij}y_j\mathbf t^{\mathbf w_j}, \qquad 1\leq i\leq n, \]
where $\bf{y}=(y_1,\ldots,y_N)$ and the remaining operator becomes
\[T=y_0\partial_{y_0}+\pi y_0+B, \qquad B=z\partial_z+\sum_{j=1}^N y_j\partial_{y_j} +\pi\sum_{j=1}^N y_j\mathbf t^{\mathbf w_j}. \]
Put \[ Q_\bullet = K_\bullet\bigl(S[\mathbf y]; \{E_{i,\gamma}(\mathbf y)\}_{i=1}^n\bigr) \otimes_K K[z^{\pm1}]. \]
The operator $B$ commutes with each $E_{i,\gamma}(\mathbf y)$. Hence $T$ defines an endomorphism of $Q_\bullet[y_0]$.

\textbf{Claim 3.} The map
\[ T:Q_\bullet[y_0]\longrightarrow Q_\bullet[y_0] \]
is degreewise injective, and
\[ \operatorname{Cone}(T)\cong K_\bullet, \qquad \operatorname{Coker}(T)\cong Q_\bullet. \]
In particular, $K_\bullet\simeq Q_\bullet$ in $D(K)$.

For $\eta\in Q_k$ and $m\geq0$, we have
\[ T(y_0^m\eta) = \pi y_0^{m+1}\eta+y_0^m(m+B)\eta.\]
Thus, for every nonzero $f\in Q_k[y_0]$, $\deg_{y_0}(Tf)=\deg_{y_0}(f)+1$, which proves injectivity. In the cokernel, the same equality gives
\[ [y_0^{m+1}\eta] = -\pi^{-1}[y_0^m(m+B)\eta]. \]
Repeatedly applying this relation, every class has a representative in $Q_k$. This representative is unique, since a nonzero element in the image of $T$ has positive $y_0$-degree. Therefore the natural inclusion induces an isomorphism
\[ Q_k\xrightarrow{\sim}Q_k[y_0]/T(Q_k[y_0]). \]
These isomorphisms commute with the differentials, giving $ Q_\bullet\xrightarrow{\sim}\operatorname{Coker}(T)$.

Adding the Koszul operator $T$ identifies $K_\bullet$ with
$\operatorname{Cone}(T)$; see the construction in the proof of \cite[Proposition 2.4.3]{FuAG2006}. Since $T$ is degreewise injective, the canonical projection
\[ \operatorname{Cone}(T)\longrightarrow\operatorname{Coker}(T) \]
is a quasi-isomorphism. This proves Claim 3.

Combining Claims 2 and 3, and renaming $\mathbf y$ as $\mathbf x$, we obtain
\[ 0=H_k(K_\bullet)\cong H_k(Q_\bullet)\cong H_k\bigl(L;\{E_{i,\gamma}\}_{i=1}^n\bigr) \otimes_K K[z^{\pm1}] \qquad(k>0).\]
Since $K[z^{\pm1}]$ is faithfully flat over $K$, this proves
\[ H_k\bigl(L;\{E_{i,\gamma}\}_{i=1}^n\bigr)=0 \qquad(k>0). \]

Finally, we prove that $D^\dagger$ is flat as a right $D$-module. By definition, \[\c D_{\s P^N,\b Q}^\dagger(\infty_N)=\varinjlim_m\b{Q}\otimes_{\b{Z}}\widehat{\c{B}}^{(m)}(\infty_N)\,\widehat{\otimes}\,\widehat{\c{D}}^{(m)}_{\s P^N}.\]
Then by the discussion in \cite[\S4.3, p.~33]{NootHuyghe2007}, we have
$D^\dagger=\varinjlim_m \b{Q}\otimes_{\b{Z}}(B^{(m)}\otimes D^{(m)})^{\wedge}$, where $B^{(m)}\otimes D^{(m)}$ are the global sections of $\c{B}^{(m)}(\infty_N)\otimes\c{D}^{(m)}_{\s P^N}$.
Note that $\b{Q}\otimes_{\b{Z}}(B^{(m)}\otimes D^{(m)})$ is just our $D$ for any $m\geq 0$. We obtain that $\b{Q}\otimes_{\b{Z}}(B^{(m)}\otimes D^{(m)})^{\wedge}$ is flat over $D$ for any $m\geq 0$ since completion is flat. It follows that $D^\dagger$ is flat over $D$.
\end{proof}

Suppose that $\Sigma_0$ is a subgroup of $\mathbb Z_p/\mathbb Z$ which does not contain $p$-adic Liouville numbers.
Let $D^b_{q\text{-}\Sigma_0} \bigl(\mathcal D^\dagger_{\mathscr X,\mathbb Q}({}^\dagger T)\bigr)$
be Caro's category of quasi $\Sigma_0$-unipotent complexes localized along
$T$; see \cite[Definition~3.3.1 and \S3.3.5]{CaroUnipotentmonodromy}.
By \cite[Theorem~1.4.15 and Remark~1.3.3]{CaroUnipotentmonodromy},
this is a triangulated subcategory of $D^b_{\mathrm{ovhol}}
\bigl(\mathcal D^\dagger_{\mathscr X,\mathbb Q}({}^\dagger T)\bigr)$.
Since tensor products and shifts preserve quasi $\Sigma_0$-unipotence,
the twisted tensor product induces a bifunctor
\[ -\widetilde\otimes_{(\mathscr{X},T)}- :
D^b_{q\text{-}\Sigma_0}\bigl(\mathcal D^\dagger_{\mathscr{X},\mathbb Q}({}^\dagger T)\bigr)
\times D^b_{q\text{-}\Sigma_0} \bigl(\mathcal D^\dagger_{\mathscr{X},\mathbb Q}({}^\dagger T)\bigr)
\longrightarrow D^b_{q\text{-}\Sigma_0} \bigl( \mathcal D^\dagger_{\mathscr{X},\mathbb Q}({}^\dagger T) \bigr). \]
The corresponding categories over couples admit Grothendieck's six operations
as in \cite[\S2.2.6]{CaroUnipotentmonodromy},
with direct images taken along complete morphisms of couples.

\begin{proof}[Proof of \Cref{complex}(i)]
   The concentration of $C_\gamma^\bullet(\c L^\dagger)$ in degree $n$
follows from \Cref{Hk=0}.
It remains to prove overholonomicity under the additional assumption
that $\Sigma_\gamma$ contains no $p$-adic Liouville numbers,
which we impose for the remainder of the proof.
 
    By \Cref{CL}, the proof of \Cref{Hypisodagger} and \cite[Lemma 5.3]{padictorus}, we have \[C_\gamma^{\bullet}(\c L^\dagger)\cong R\rho'_{2,\ast}\m{sp}_\ast\Big((j'\times j_N)^\dagger \Omega^\bullet_{X_K(\Sigma')^{\m{an}}\times \b P^{N,{\m{an}}}_K/\b P^{N,{\m{an}}}_K}\big(\m{log} (D'\times \b P^N)\big) \otimes \rho_1'^\ast \widetilde{\c{K}}_\gamma\otimes \widetilde{\c L}_{\pi,\overline F}\Big).\]
    Let $D'':=X_{\kappa}(\Sigma')-X_{\delta,\kappa}'$, let $B:=(D''\times\b{P}^N_{\kappa})\cup(X_{\kappa}(\Sigma')\times\infty_N)$, and let $\overline{D'}$ be the closure of $D'$ in $X_{\kappa}(\Sigma')$. 
    We have the following canonical quasi-isomorphism
    \begin{equation}\begin{aligned}\label{DRquasiiso}
        & \,\m{sp}_\ast\Big((j'\times j_N)^\dagger \Omega^\bullet_{X_K(\Sigma')^{\m{an}}\times \b P^{N,{\m{an}}}_K/\b P^{N,{\m{an}}}_K}\big(\m{log} (D'\times \b P^N)\big) \otimes \rho_1'^\ast \widetilde{\c{K}}_\gamma\otimes \widetilde{\c L}_{\pi,\overline F}\Big)
        \\ \cong & \,\Omega^\bullet_{\s{X}(\Sigma')\times\s{P}^N/\s{P}^N,\b{Q}} \otimes_{\c{O}_{\s{X}(\Sigma')\times\s{P}^N,\b{Q}}} u_{B,+}\big(\m{sp}_\ast\rho_1'^\ast \widetilde{\c{K}}_\gamma\big)\otimes \m{sp}_\ast\widetilde{\c L}_{\pi,\overline F}
    \end{aligned}\end{equation}
    by \cite[Proposition 2.2.11]{CaroTsuzuki12}, where $u:(\s{X}(\Sigma')\times\s{P}^N,\overline{D'}\times \b P^N)\to \s{X}(\Sigma')\times\s{P}^N$ is the morphism forgetting the log-structure, $u_{B,+}$ is defined in \cite[5.22]{Carologisocrystal}.
    Then we consider another log-structure and the morphism forgetting the log-structure \[v:(\s{X}(\Sigma')\times\s{P}^N,H_{\Sigma'}\times \b P^N)\to \s{X}(\Sigma')\times\s{P}^N.\]
    Define the global Kummer log-isocrystal by \[\widetilde{\c K}_\gamma^{\log}:=\left(\mathcal O_{\s{X}(\Sigma')\times\s{P}^N,\mathbb Q}\cdot e,\,\nabla e=\sum_{i=1}^n\gamma_i\frac{dt_i}{t_i}\cdot e\right).\]
    We obtain $(^\dagger B)\widetilde{\c K}_\gamma^{\log}\cong \m{sp}_\ast\rho_1'^\ast \widetilde{\c{K}}_\gamma$.
    It follows that \[u_{B,+}\big(\m{sp}_\ast\rho_1'^\ast \widetilde{\c{K}}_\gamma\big)\cong v_{B,+}\big((^\dagger B)\widetilde{\c K}_\gamma^{\log}\big)\cong (^\dagger B)v_+(\widetilde{\c K}_\gamma^{\log})\] by \cite[Lemma 2.2.2]{CaroTsuzuki12}.
    One checks that all the exponents of $\widetilde{\c K}_\gamma^{\log}$ lie in $\overline\Sigma_{\gamma}$, where $\overline\Sigma_{\gamma}:=(\Sigma_{\gamma}+\b{Z})/\b{Z}$. So we have $v_+(\widetilde{\c K}_\gamma^{\log})\in D_{\m{ovhol}}^b(\c{D}_{\s{X}(\Sigma')\times\s{P}^N,\b{Q}}^\dagger)$ by \cite[Theorem 2.3.13]{CaroTsuzuki12}. Moreover, since $\s{X}(\Sigma')$ has the form $\s{A}^r\times\s{T}^{n-r}$ locally, by \cite[Proposition 1.1.22]{CaroTsuzuki12} and the same argument of the proof of \cite[Proposition 2.3.4]{CaroTsuzuki12}, one checks that $v_+(\widetilde{\c K}_\gamma^{\log})\in D_{q\text{-}\overline\Sigma_{\gamma}}^b(\c{D}_{\s{X}(\Sigma')\times\s{P}^N,\b{Q}}^\dagger)$.
    Note that $\m{sp}_\ast\widetilde{\c L}_{\pi,\overline F}$ is also quasi $\overline\Sigma_{\gamma}$-unipotent since it is an overconvergent $F$-isocrystal. We obtain that $u_{B,+}\big(\m{sp}_\ast\rho_1'^\ast \widetilde{\c{K}}_\gamma\big)\otimes \m{sp}_\ast\widetilde{\c L}_{\pi,\overline F}$ is quasi $\overline\Sigma_{\gamma}$-unipotent since this property is preserved by tensor product and $(^\dagger B)$.
    Then applying $R\rho_{2,*}'$ to the quasi-isomorphism \ref{DRquasiiso}, we obtain the following quasi-isomorphism \[C_\gamma^{\bullet}(\c L^\dagger)\cong \rho''_{2,+}(u_{B,+}\big(\m{sp}_\ast\rho_1'^\ast \widetilde{\c{K}}_\gamma\big)\otimes \m{sp}_\ast\widetilde{\c L}_{\pi,\overline F})[-n]\] in the category $D_{\m{coh}}^b(\c{D}_{\s{P}^N,\b{Q}}^\dagger(\infty_N))$, where $\rho_2''$ is the canonical projection $\s{X}(\Sigma')\times\s{P}^N\to \s{P}^N$, and we have $R\rho_{2,*}'=R\rho_{2,*}''$. 
    In summary, $C_\gamma^{\bullet}(\c L^\dagger)$ is quasi $\overline\Sigma_{\gamma}$-unipotent, hence overholonomic.
\end{proof}

\subsection{The Frobenius structure on $H^n(C_{\gamma}^{\bullet}(L^\dagger))$}

Recall that $\delta=\{\bf{v}\,|\,H_1(\bf{v})\geq 0, \cdots,H_s(\bf{v})\geq 0\}$. We define $\delta(\bf{d})=\{\bf{v}\,|\,H_1(\bf{v})\geq -d_1, \cdots,H_s(\bf{v})\geq -d_s\}$ for any $\bf{d}=(d_1,\cdots,d_s)\in \b{Z}^s$. Let \[L^{\dagger}(\bf{d}):=\Big\{\sum_{\mathbf{u}\in\b{Z}^N_{\geq 0},\mathbf{v}\in \delta(\bf{d})\cap \b Z^n}a_{\mathbf{u},\mathbf{v}}\mathbf{x}^{\mathbf{u}}\bf t^{\mathbf{v}}\,\Big|\,a_{\mathbf{u},\mathbf{v}}\in K,\,\exists(c>0,0<\eta<1),\|a_{\mathbf{u},\mathbf{v}}\|\leq c\eta^{|\mathbf{u}|+|\mathbf{v}|}\Big\}.\] One verifies that the differential operators $E_{i,\gamma}$ are well defined on $L^{\dagger}(\bf{d})$. We thus define the complex $C_\gamma^{\bullet}(L^\dagger(\bf{d}))$ by the same construction in \Cref{C}.

\begin{prop}\label{shiftdcoh}
    $L^{\dagger}(\bf{d})$ is a coherent $\Gamma(\s P^N, \mathcal D^\dagger_{\mathscr{P}^N, \mathbb Q}(\infty_N))$-module. Hence it defines a coherent $\c D^\dagger_{\s P^N,\b Q}(\infty_N)$-module, which we denote by $\c L^\dagger(\bf{d})$.
    Moreover, if $\Sigma_{\gamma}$ does not have $p$-adic Liouville numbers, then $C_\gamma^{\bullet}(\c L^\dagger(\bf{d}))$ is an overholonomic complex.
\end{prop}
\begin{proof}
    Let $S=K[\bf{t}^{\bf{v}}:\bf{v}\in\delta\cap\b{Z}^n]$. It is a Noetherian ring. Let $S(\bf{d}):=\oplus_{\bf{v}\in\delta(\bf{d})\cap\b{Z}^n} K\bf{t}^{\bf{v}}$ be an $S$-module. Choose $\bf{u}\in\b{Z}^n$ such that $H_b(\bf{u})\geq d_b$ for any $b=1,\ldots,s$. Then $\bf{t}^{\bf{u}}S(\bf{d})$ is an ideal of $S$. Consequently, $\bf{t}^{\bf{u}}S(\bf{d})$, and hence $S(\bf{d})$, are finitely generated as $S$-modules.
    Then choose $\bf{v}_1,\ldots,\bf{v}_r$ such that $S(\bf{d})=\sum_{\alpha=1}^r\bf{t}^{\bf{v}_{\alpha}}S$. One checks that $L^\dagger(\bf{d})=\sum_{\alpha=1}^r\bf{t}^{\bf{v}_{\alpha}}L^\dagger$. It follows that $L^\dagger(\bf{d})$ can be regarded as a quotient of $(L^{\dagger})^{\oplus r}$. So it is coherent.

    To show $C_\gamma^{\bullet}(\c L^\dagger(\bf{d}))$ is an overholonomic complex, we describe the argument above more precisely.
    Recall the rings $D,D^{\dagger}$ defined in the proof of \Cref{CL}. We have $L^\dagger\cong D^{\dagger}\otimes_DK[\bf{x}]\otimes_KS$. Hence $L^\dagger(\bf{d})\cong L^\dagger\otimes_SS(\bf{d})$. Also note that $D^{\dagger}$ is flat over $D$, which is proved in \Cref{Hk=0}. We obtain that $L^\dagger$ is flat over $S$.
    Consider the surjective morphism 
    \[\oplus_{\alpha=1}^rL^{\dagger}e_{\alpha}\to L^\dagger(\bf{d}),\quad e_{\alpha}\mapsto \bf{t}^{\bf{v}_{\alpha}}.\]
    Note that $t_i\partial_{t_i}(\bf{t}^{\bf{v}_{\alpha}})=v_{{\alpha},i}\bf{t}^{\bf{v}_{\alpha}}$. For any $\gamma\in\b{Z}_p^n$, it induces a surjective morphism \[\oplus_{\alpha=1}^rC_{\gamma+\bf{v}_{\alpha}}^{\bullet}(L^{\dagger})e_{\alpha}\to C_\gamma^{\bullet}(L^\dagger(\bf{d}))\] between coherent $\Gamma(\s P^N, \mathcal D^\dagger_{\mathscr{P}^N, \mathbb Q}(\infty_N))$-complexes.
    Hence we obtain a surjective morphism \[\oplus_{\alpha=1}^rC_{\gamma+\bf{v}_{\alpha}}^{\bullet}(\c L^{\dagger})e_{\alpha}\to C_\gamma^{\bullet}(\c L^\dagger(\bf{d}))\] between coherent $\mathcal D^\dagger_{\mathscr{P}^N, \mathbb Q}(\infty_N)$-complexes.
    Then consider a free of finite rank resolution for the $S$-module $S(\bf{d})$.
    Applying $L^\dagger\otimes_S-$, by the same argument, we can obtain a resolution of $C_\gamma^{\bullet}(\c L^\dagger(\bf{d}))$ (not necessarily of finite length) all of whose terms are finite direct sums of $C_{\gamma+\bf{v}}^{\bullet}(\c L^{\dagger})$ for some $\bf{v}\in \b{Z}^n$.
    By our assumption, $\Sigma_{\gamma}$ does not have $p$-adic Liouville numbers. It follows that $C_{\gamma+\bf{v}}^{\bullet}(\c L^{\dagger})\in D_{\m{ovhol}}^b(\mathcal D^\dagger_{\mathscr{P}^N, \mathbb Q}(\infty_N))$ by \Cref{complex}(i).
    Note that $C_\gamma^{\bullet}(\c L^\dagger(\bf{d}))$ is bounded by definition and the category of overholonomic objects is a Serre subcategory. We have $C_{\gamma}^{\bullet}(\c L^{\dagger}(\bf{d}))\in D_{\m{ovhol}}^b(\mathcal D^\dagger_{\mathscr{P}^N, \mathbb Q}(\infty_N))$ by the canonical spectral sequence argument.
\end{proof}

\begin{prop}\label{shiftbyfacet}
    Fix $1\leq b\leq s$. Let $\bf{1}_b\in\b{Z}^s$ be the vector with $1$ in the $b$-th coordinate and $0$ otherwise. For any $\bf{d}=(d_1,\cdots,d_s)\in \b{Z}^s$, the canonical inclusion $C_\gamma^{\bullet}(L^\dagger(\bf{d}))\to C_\gamma^{\bullet}(L^\dagger(\bf{d}+\bf{1}_b))$ is a quasi-isomorphism if $H_b(\gamma)\neq d_b+1$.

    Assume that $H_b(\gamma)= d_b+1$ and $\Sigma_{\gamma}$ does not have $p$-adic Liouville numbers. Let $i_{\bf{a}}:\{\bf{a}\}\to\s{P}^N$ be a closed point, and let $K_{\bf{a}}$ be the field of $\bf{a}$. We still have $\chi_{\bf{a}}\big(C_\gamma^{\bullet}(\c L^\dagger(\bf{d}))\big)=\chi_{\bf{a}}\big(C_\gamma^{\bullet}(\c L^\dagger(\bf{d}+\bf{1}_b))\big)$, where \[\chi_{\bf{a}}\big(\c{E}^\bullet\big):=\sum(-1)^k\dim_{K_{\bf{a}}}H^k\big(i_{\bf{a}}^!\c{E}^{\bullet}\big).\] It is well defined for overholonomic complexes.
\end{prop}
\begin{proof}
    Denote $L^\dagger(\bf{d}+\bf{1}_b)/L^\dagger(\bf{d})$ by $Q_{\bf{d},b}^\dagger$ for short. We have \[\m{Cone}\big(C_\gamma^{\bullet}(L^\dagger(\bf{d}))\to C_\gamma^{\bullet}(L^\dagger(\bf{d}+\bf{1}_b))\big)\cong C_\gamma^{\bullet}\big(Q_{\bf{d},b}^\dagger\big)\] in the derived category. So for the first assertion, it suffices to show that $C_\gamma^{\bullet}\big(Q_{\bf{d},b}^\dagger\big)$ is acyclic.
    Still let $H_b(z)=\sum_{i=1}^n c_i z_i\in \b Z[z_1,\cdots, z_n]$ be the defining linear form for the hyperplane. Consider the $K$-linear morphism
    \begin{align*}
   h^k& :C_\gamma^k\big(Q_{\bf{d},b}^\dagger\big) \to C_\gamma^{k-1} \big(Q_{\bf{d},b}^\dagger\big),\quad
    \\ & \xi\frac{dt_{i_1}}{t_{i_1}}\wedge\cdots\wedge\frac{dt_{i_k}}{t_{i_k}}\mapsto\sum_{j=1}^k(-1)^{j-1}c_{i_j}\xi\frac{dt_{i_1}}{t_{i_1}}\wedge\cdots\wedge\widehat{\frac{dt_{i_j}}{t_{i_j}}}\wedge\cdots\wedge\frac{dt_{i_k}}{t_{i_k}}.
    \end{align*} One checks that $h^k$ is well defined and every class in the quotient has a representative of the form $\xi=\sum_{H_b(\bf{v})=-d_b-1}a_{\mathbf{u},\mathbf{v}}\mathbf{x}^{\mathbf{u}}\bf t^{\mathbf{v}}$. A direct computation gives
{\small
\begin{eqnarray*}(h^{k+1}d_{\gamma}^k+d_{\gamma}^{k-1}h^k)(\xi\frac{dt_{i_1}}{t_{i_1}}\wedge\cdots\wedge\frac{dt_{i_k}}{t_{i_k}})&=&\big(\sum_{i=1}^nc_it_i\partial_{t_i}+H_b(\gamma)+\sum_{j=1}^N\pi H_b(\bf{w}_j)x_j\bf{t}^{\bf{w}_j}\big)(\xi\frac{dt_{i_1}}{t_{i_1}}\wedge\cdots\wedge\frac{dt_{i_k}}{t_{i_k}})\\
&=&\big(\sum_{i=1}^nc_it_i\partial_{t_i}+H_b(\gamma)\big)(\xi\frac{dt_{i_1}}{t_{i_1}}\wedge\cdots\wedge\frac{dt_{i_k}}{t_{i_k}})\\
&=&(H_b(\gamma)-d_b-1)(\xi\frac{dt_{i_1}}{t_{i_1}}\wedge\cdots\wedge\frac{dt_{i_k}}{t_{i_k}}).
\end{eqnarray*}
}For the second equality, if $H_b(\bf{w}_j)=0$, the corresponding term vanishes, while if $H_b(\bf{w}_j)>0$, then $\bf{t}^{\bf{w}_j}\xi\in L^\dagger(\bf{d})$ and its class in the quotient is $0$. The third equality follows from the fact that $\xi$ is homogeneous of $H_b$-degree $-d_b-1$. Consequently, these $(H_b(\gamma)-d_b-1)^{-1}h^k$ induce a null-homotopy on the complex $C_\gamma^{\bullet}(Q_{\bf{d},b}^\dagger)$ when $H_b(\gamma)\neq d_b+1$. This proves the first assertion.

By \Cref{shiftdcoh}, $Q_{\bf{d},b}^\dagger$ is a coherent $\Gamma(\s P^N, \mathcal D^\dagger_{\mathscr{P}^N, \mathbb Q}(\infty_N))$-module. Hence it defines a coherent $\c D^\dagger_{\s P^N,\b Q}(\infty_N)$-module, which we denote by $\c Q_{\bf{d},b}^\dagger$. Then let $K_\bullet(\c Q_{\bf{d},b}^\dagger;\{\overline{E}_{1,\gamma},\cdots,\overline{E}_{n,\gamma}\})$ be the homological Koszul complex of the commuting operators $\overline{E}_{1,\gamma},\cdots,\overline{E}_{n,\gamma}$ on $\c Q_{\bf{d},b}^\dagger$. We have 
$H^k\big(C_\gamma^{\bullet}\big(\c Q_{\bf{d},b}^\dagger\big)\big)=H_{n-k}\big(K_\bullet(\c Q_{\bf{d},b}^\dagger;\{\overline{E}_{1,\gamma},\cdots,\overline{E}_{n,\gamma}\})\big)$.
If $H_b(\gamma)= d_b+1$, we have $\sum_{i=1}^nc_i\overline{E}_{i,\gamma}=0$. Then after renumbering the coordinates, we may assume that $c_1\neq 0$. Consider the following change of generators \[(\overline{E}_{1,\gamma},\cdots,\overline{E}_{n,\gamma})\begin{pmatrix}
    & & & c_1 \\ 1 & & & c_2 \\ & \ddots & & \vdots \\ & & 1 & c_n
\end{pmatrix}=(\overline{E}_{2,\gamma},\cdots,\overline{E}_{n,\gamma},0).\] We obtain an isomorphism of complexes \[K_\bullet(\c Q_{\bf{d},b}^\dagger;\{\overline{E}_{1,\gamma},\cdots,\overline{E}_{n,\gamma}\})\cong K_\bullet(\c Q_{\bf{d},b}^\dagger;\{\overline{E}_{2,\gamma},\cdots,\overline{E}_{n,\gamma},0\}).\]
Hence it is isomorphic to $K_\bullet(\c Q_{\bf{d},b}^\dagger;\{\overline{E}_{2,\gamma},\cdots,\overline{E}_{n,\gamma}\})\oplus K_\bullet(\c Q_{\bf{d},b}^\dagger;\{\overline{E}_{2,\gamma},\cdots,\overline{E}_{n,\gamma}\})[-1]$.
It follows that \begin{align*}
    &\,\chi_{\bf{a}}\big(K_\bullet(\c Q_{\bf{d},b}^\dagger;\{\overline{E}_{1,\gamma},\cdots,\overline{E}_{n,\gamma}\})\big)
    \\ = & \,\chi_{\bf{a}}\big(K_\bullet(\c Q_{\bf{d},b}^\dagger;\{\overline{E}_{2,\gamma},\cdots,\overline{E}_{n,\gamma}\})\big)-\chi_{\bf{a}}\big(K_\bullet(\c Q_{\bf{d},b}^\dagger;\{\overline{E}_{2,\gamma},\cdots,\overline{E}_{n,\gamma}\})\big)=0.
\end{align*}
This proves the second assertion.
\end{proof}

\begin{cor}\label{shiftbycone}
    Fix $1\le j\le N$. If $H_b(\gamma)\notin \{1,2,\ldots,H_b(\bf{w}_j)\}$ for any $1\le b\le s$, then the morphism $\bf{t}^{\bf{w}_j}:C_{\gamma}^{\bullet}(L^{\dagger})\to C_{\gamma-\bf{w}_j}^{\bullet}(L^{\dagger})$ is a quasi-isomorphism.

    If $\Sigma_{\gamma}$ does not have $p$-adic Liouville numbers, we still have $\chi_{\bf{a}}\big(C_\gamma^{\bullet}(\c L^\dagger)\big)=\chi_{\bf{a}}\big(C_{\gamma-\bf{w}_j}^{\bullet}(\c L^\dagger)\big)$ for any closed point $\bf{a}\in\s{P}^N$.
\end{cor}
\begin{proof}
By \Cref{shiftbyfacet}, it follows that $C_\gamma^{\bullet}(L^\dagger)\to C_\gamma^{\bullet}(L^\dagger(\bf{d}_j))$ is a quasi-isomorphism, where $\bf{d}_j=(H_1(\bf{w}_j),\cdots,H_s(\bf{w}_j))$. We also have $\chi_{\bf{a}}\big(C_\gamma^{\bullet}(\c L^\dagger)\big)=\chi_{\bf{a}}\big(C_\gamma^{\bullet}(\c L^\dagger(\bf{d}_j))\big)$ if $\Sigma_{\gamma}$ does not have $p$-adic Liouville numbers. Then the assertion follows from the following commutative diagram \[\begin{tikzcd}
    L^\dagger(\bf{d}_j) \ar[rrrr,"t_i\partial_{t_i}+\gamma_i+\sum_{r=1}^N\pi w_{ir}x_r\bf{t}^{\bf{w}_r}"]\ar[d,"\cong ","\cdot \bf{t}^{\bf{w}_j}"'] & & & & L^\dagger(\bf{d}_j)\ar[d,"\cong"',"\cdot \bf{t}^{\bf{w}_j}"] \\ L^\dagger\ar[rrrr,"t_i\partial_{t_i}+\gamma_i-w_{ij}+\sum_{r=1}^N\pi w_{ir}x_r\bf{t}^{\bf{w}_r}"] & & & & L^\dagger. \end{tikzcd}\qedhere\]
\end{proof}

If $\sum_{j=1}^N\b{Z}_{\geq 0}\bf{w}_j=\delta\cap\b{Z}^n$, then the first assertion of the above corollary is the $p$-adic analog of \cite[Remark 3.7]{SW09}. The equivalence of the two conditions can be found in \cite[Lemma 3.8]{Fang}.

\begin{cor}\label{Frobenius}
    Assume $(q-1)\gamma\in \b{Z}^n$. Then we have a canonical Frobenius structure on $C_\gamma^{\bullet}(\c L^\dagger)\in D_{\m{coh}}^b(\c D^{\dagger}_{\s{P}^N,\b{Q}}(\infty_N))$.
\end{cor}
\begin{proof}
    Since $C_\gamma^\bullet(\c L^\dagger)$ is coherent, it suffices to give a Frobenius structure on the global sections $C_\gamma^\bullet(L^\dagger)$. Consider the $K$-algebra homomorphism $\m{Fr}^{\sharp}:K\left<\bf{x}\right>^{\dagger}\to K\left<\bf{x}\right>^{\dagger},\,\,x_j\mapsto x_j^q$. Let
$$(C_\gamma^{(1),\bullet}(L^\dagger),\nabla^{(1)}):=\m{Fr}^*(C_\gamma^{\bullet}(L^\dagger), \nabla)=(C_\gamma^{\bullet}(L^\dagger),\nabla)\otimes_{K\left<\bf{x}\right>^{\dagger},\m{Fr}^{\sharp}}K\left<\bf{x}\right>^{\dagger}.$$
As in \cite[Section 2.1]{padictorus}, one verifies that $C_\gamma^{(1),\bullet}(L^\dagger)$ is the logarithmic de Rham complex with $C_\gamma^{(1),k}(L^\dagger)=C_\gamma^{k}(L^\dagger)$ for each $k$, and the differential $d^{(1),k}_\gamma:C_\gamma^{(1),k}(L^\dagger)\to C_\gamma^{(1),k+1}(L^\dagger)$ is given by
$$d^{(1),k}_\gamma(\omega)=\mathrm d_{\mathbf t}\omega + \sum_{i=1}^n \Big(\gamma_i+\sum_{j=1}^N \pi w_{ij}x_j^q\mathbf t^{\mathbf w_j}\Big)\frac{\mathrm dt_i}{t_i}\wedge \omega.
$$
The connection $\nabla^{(1)}$ is defined by
$$\nabla^{(1)}(\omega)=d_{\bf{x}}\omega+\sum_{j=1}^N q\pi x_j^{q-1}\mathbf t^{\mathbf w_j}dx_j\wedge\omega.$$

We first construct a horizontal quasi-isomorphism $F_1:(C_\gamma^{(1),\bullet}(L^\dagger),\nabla^{(1)})\to (C_{q\gamma}^{\bullet}(L^\dagger), \nabla)$. Let $\Phi$ and $\Psi$ be operators in the variables $\mathbf t$ defined by
$$\Phi(\sum_{\bf{v}} a_{\bf{v}}(\bf{x})\bf{t}^{\bf{v}})=\sum_{\bf{v}} a_{\bf{v}}(\bf{x})\bf{t}^{q\bf{v}}, \ \Psi(\sum_{\bf{v}} a_{\bf{v}}(\bf{x})\bf{t}^{\bf{v}})=\sum_{\bf{v}} a_{q\bf{v}}(\bf{x})\bf{t}^{\bf{v}}.$$
We extend $\Phi, \Psi$ to differential forms by
\begin{eqnarray*}
\Phi\Big(\sum_{1\leq i_1<\cdots < i_k\leq n}f_{i_1\ldots i_k}(\mathbf x, \mathbf t)\frac{\mathrm dt_{i_1}}{t_{i_1}}\wedge \cdots\wedge \frac{\mathrm dt_{i_k}}
{t_{i_k}}\Big)
&=&\sum_{1\leq i_1<\cdots < i_k\leq n}q^kf_{i_1\ldots i_k}(\mathbf x, \mathbf t^q)\frac{\mathrm dt_{i_1}}{t_{i_1}}\wedge \cdots\wedge 
\frac{\mathrm dt_{i_k}}{t_{i_k}},\\
\Psi\Big(\sum_{1\leq i_1<\cdots < i_k\leq n}f_{i_1\ldots i_k}(\mathbf x, \mathbf t)\frac{\mathrm dt_{i_1}}{t_{i_1}}\wedge \cdots\wedge \frac{\mathrm dt_{i_k}}
{t_{i_k}}\Big)
&=&\sum_{1\leq i_1<\cdots < i_k\leq n}q^{-k} \Psi(f_{i_1\ldots i_k}(\mathbf x,\mathbf t))\frac{\mathrm dt_{i_1}}{t_{i_1}}\wedge \cdots\wedge 
\frac{\mathrm dt_{i_k}}{t_{i_k}}.
\end{eqnarray*}
They commute with $d_{\bf{t}}$. For each $k$, define the maps $$F_1:C^{(1),k}_\gamma(L^\dagger)\to C^k_{q\gamma}(L^\dagger), \ G_1:C^k_{q\gamma}(L^\dagger)\to C^{(1),k}_\gamma(L^\dagger)$$ by
\begin{eqnarray*}
F_1&=&\exp(-\pi F(\bf{x},\bf{t}))\circ \Phi\circ \exp(\pi F(\bf{x}^q,\bf{t}))=\exp(\pi F(\bf{x}^q,\bf{t}^q)-\pi F(\bf{x},\bf{t}))\circ\Phi,\\
G_1&=&\exp(-\pi F(\bf{x}^q,\bf{t}))\circ \Psi\circ \exp(\pi F(\bf{x},\bf{t}))=\Psi\circ \exp(\pi F(\bf{x},\bf{t})-\pi F(\bf{x}^q,\bf{t}^q)).
\end{eqnarray*}
Both exponential functions on the right hand sides belong to $L^\dagger$, so they are well defined. One verifies directly that $$F_1\circ d^{(1),k}_\gamma=d^{k}_{q\gamma}\circ F_1, \ F_1\circ \nabla^{(1)}=\nabla \circ F_1,$$ $$G_1\circ d^{k}_{q\gamma}=d^{(1),k}_\gamma\circ G_1, \ G_1\circ \nabla=\nabla^{(1)}\circ G_1.$$ Then 
$$F_1:(C_\gamma^{(1),\bullet}(L^\dagger),\nabla^{(1)})\to (C_{q\gamma}^{\bullet}(L^\dagger), \nabla), \ G_1:(C_{q\gamma}^{\bullet}(L^\dagger), \nabla)\to (C_\gamma^{(1),\bullet}(L^\dagger),\nabla^{(1)}) $$ 
are horizontal morphisms of complexes.  We have $G_1\circ F_1=\text{id}$ on $C^{(1),\bullet}_\gamma(L^\dagger)$. Let $K'/K$ be a finite extension containing all $q$-th roots of unity. Applying the homotopy argument of \cite[Section 2.2]{padictorus} to the parameter $q\gamma$, we obtain that $F_1\circ G_1$ is homotopic to $\m{id}$ on $C^\bullet_{q\gamma}(L^\dagger)\otimes_K K'$. Since $K'/K$ is faithfully flat, $F_1$ and $G_1$ induce isomorphisms on cohomology groups inverse to each other.

Next, we construct a horizontal quasi-isomorphism $(C_{q\gamma}^{\bullet}(L^\dagger),\nabla)\to (C_{\gamma}^{\bullet}(L^\dagger), \nabla)$. For $1\le b\le s$, let 
$$d_b=H_b((q-1)\gamma)\in\b{Z},\ d_b^+=\max\{0,d_b\},\  d_b^-=\max\{0,-d_b\}.$$
Consider the vectors
 $$\bf{d}^+=(d_b^+)_{1\le b\le s}\in \b Z_{\ge 0}^s,\ \bf{d}^-=(d_b^-)_{1\le b\le s}\in\b{Z}_{\ge 0}^s.$$
Then $\bf{d}:=(d_b)_{1\le b\le s}=\bf{d}^+ - \bf{d}^-.$ For any $1\le b \le s$ such that $d_b=d_b^+>0$, we have $H_b(\gamma)=\frac{d_b}{q-1}>0$. Hence
$$H_b(q\gamma)=d_b+H_b(\gamma)\notin \{1,2,\ldots, d_b\}. $$
By successive applications of \Cref{shiftbyfacet}, the inclusion $(C_{q\gamma}^{\bullet}(L^\dagger),\nabla)\to (C_{q\gamma}^{\bullet}(L^\dagger(\bf{d}^+)), \nabla)$ is a horizontal quasi-isomorphism. For any $1\le b\le s$ such that $d_b=-d_b^-< 0$, we have $H_b(\gamma)=\frac{d_b}{q-1}< 0$,
and hence $$H_b(\gamma)\notin\{1,2,\ldots, d_b^-\}.$$
Again by \Cref{shiftbyfacet}, the inclusion $(C_{\gamma}^{\bullet}(L^\dagger),\nabla)\to (C_{\gamma}^{\bullet}(L^\dagger(\bf{d}^-)),\nabla)$ is a horizontal quasi-isomorphism. 
Since $\delta(\bf{d}^+)+(q-1)\gamma=\delta(\bf{d}^-),$ multiplication by $\bf{t}^{(q-1)\gamma}$ defines an isomorphism of modules:
$$\bf{t}^{(q-1)\gamma}\cdot: L^\dagger(\bf{d}^+)\to L^\dagger(\bf{d}^-),$$ with inverse given by multiplication by $\bf{t}^{(1-q)\gamma}$. It extends to a horizontal isomorphism
$$\bf{t}^{(q-1)\gamma}\cdot: (C_{q\gamma}^{\bullet}\big(L^\dagger(\bf{d}^+)),\nabla\big)\to \big(C_{\gamma}^{\bullet}(L^\dagger(\bf{d}^-)), \nabla\big).$$
This is because multiplication by $\bf{t}^{(q-1)\gamma}$ commutes with the connection operators and the following diagram commutes for any $i$:
\[\begin{tikzcd} L^\dagger(\bf{d}^+) \ar[rrrr,"t_i\partial_{t_i}+q\gamma_i+\sum_{j=1}^N\pi w_{ij}x_j\bf{t}^{\bf{w}_j}"]\ar[d,"\cong ","\cdot \bf{t}^{(q-1)\gamma}"'] & & & & L^\dagger(\bf{d}^+)\ar[d,"\cong"',"\cdot \bf{t}^{(q-1)\gamma}"] \\ L^\dagger(\bf{d}^-)\ar[rrrr,"t_i\partial_{t_i}+\gamma_i+\sum_{j=1}^N\pi w_{ij}x_j\bf{t}^{\bf{w}_j}"] & & & & L^\dagger(\bf{d}^-). \end{tikzcd}\] 
In summary, we construct the horizontal isomorphisms
$$(C_{q\gamma}^{\bullet}(L^\dagger),\nabla)\xrightarrow{\cong} \big(C_{q\gamma}^{\bullet}(L^\dagger(\bf{d}^+)), \nabla\big)\xrightarrow{\bf{t}^{(q-1)\gamma}\cdot} \big(C_{\gamma}^{\bullet}(L^\dagger(\bf{d}^-)), \nabla\big)\xrightarrow{\cong} (C_{\gamma}^{\bullet}(L^\dagger),\nabla).$$
Together with $F_1$, this yields the required Frobenius structure in the derived category. This completes the proof.
\end{proof}

\begin{proof}[Proof of \Cref{complex}(iii)]
    We prove this by the same argument as in the proof of \cite[Proposition 1.18]{padictorus}.
    First, we assume that $(q-1)\gamma\in (-\delta)\cap\b{Z}^n$. In this case, $H_b(\gamma)\leq 0$ for any $b=1,\ldots,s$. It follows that the inverse of the Frobenius constructed in \Cref{Frobenius} is just the Dwork operator $G$ defined in \cite[Proposition 1.5]{padictorus}. Let $i_{\bf{a}}:\bf{a}\to U_0$ be a non-degenerate point. By the calculation in \cite[5.8]{padictorus}, $i_{\bf{a}}^![N]C_{\gamma}^{\bullet}(\c L^\dagger)$ is isomorphic to $C_{\gamma}^{\bullet}(L^\dagger_0)$ in \cite[5.6]{padictorus}. So the finiteness of $C_{\gamma}^{\bullet}(L^\dagger_0)$ follows from the overholonomicity of $C_{\gamma}^{\bullet}(\c L^\dagger)$. Therefore, by the same argument as in \cite[5.6]{padictorus}, we obtain that $\chi_{\bf{a}}\big(C_{\gamma}^{\bullet}(\c L^\dagger))=(-1)^{n+N}n!\operatorname{vol}(\Delta)$. 
    
    For general $\gamma$ with $(q-1)\gamma\in\b{Z}^n$, we may choose an open dense subset $V_0\subseteq \mathbb{A}_{\kappa}^N$ such that $H^n(C_{\gamma}^{\bullet}(\c L^\dagger))|_{V_0}$ is a convergent isocrystal. Therefore, by \Cref{shiftbycone}, we have $\dim(H^n(C_{\gamma}^{\bullet}(L^\dagger_0)))=n!\operatorname{vol}(\Delta)$ if $\bf{a}\in V_0\cap U_0$. Then the proof of \cite[Lemma 5.13]{padictorus} applies in our setting. Moreover, by the same argument of \cite[5.14]{padictorus}, $H^n(C_{\gamma}^{\bullet}(\c L^\dagger))|_{U_0}$ is a convergent isocrystal. Hence it is overconvergent by \cite[Theorem 6.2.8]{Caro19}. This proves the theorem.
\end{proof}

\begin{proof}[Proof of \Cref{Dmodule}]
    Note that $\m{Hyp}_{\gamma}\cong\m{Hyp}_{\gamma-\bf{v}}$ for any vector $\bf{v}\in\b{Z}^n$. By the comparison theorem \ref{main}, this is an immediate corollary of \Cref{complex}.
\end{proof}

{\sc \noindent\textbf{Data Availability}} All data included in this study are available upon request by contact with the corresponding author.
\begin{center}
{\sc \textbf{Declarations}} 
\end{center}

{\sc \noindent\textbf{Conflict of interest}} On behalf of all authors, the corresponding author states that there is no conflict of interest.

{\sc \noindent\textbf{AI use statement}} The main results and their proofs were developed by the authors. AI tools were used to assist in reading some references, reviewing the manuscript for potential errors and inconsistencies, and polishing some of the proofs.
The authors take full responsibility for the results of this paper.

\begin{center}
{\sc \textbf{Acknowledgements}}
\end{center}
\noindent The authors thank their advisor, Professor Lei Fu, for proposing the problem studied in this paper and for his valuable suggestions and helpful discussions.
H. Zhang is supported by National Key R\&D Program of China 2024YFA1014600, Jiangsu Provincial Scientific Research Center of Applied Mathematics (Grant No. BK20233002), National Natural Science Foundation of China (Grant No. 12301012), Natural Science Foundation of Jiangsu Province (Grant No. BK20230802),  and also by Zhishan Scholars
Programs of Southeast University (Grant No. 2242025RCB0011).

\bibliographystyle{alpha}
\bibliography{ref}

@article {GKZ,
    AUTHOR = {Gel'fand, I. M. and Zelevinski\u{i}, A. V. and Kapranov, M. M.},
     TITLE = {Hypergeometric functions and toral manifolds},
   JOURNAL = {Funct. Anal. Appl.},
  FJOURNAL = {Functional Analysis and Its Applications},
    VOLUME = {23},
      YEAR = {1989},
    NUMBER = {2},
     PAGES = {94--106},
      ISSN = {0016-2663},
   MRCLASS = {22E30 (14L32 32C42 33A30 58G07)},
  MRNUMBER = {1011353},
MRREVIEWER = {A.\ Verona},
       DOI = {10.1007/BF01078777},
       URL = {https://doi.org/10.1007/BF01078777},
}

@incollection {Huyghe2004Trans,
    AUTHOR = {Noot-Huyghe, Christine},
     TITLE = {Transformation de {F}ourier des {$\s{D}$}-modules
              arithm\'{e}tiques. {I}},
 BOOKTITLE = {Geometric aspects of {D}work theory. {V}ol. {I}, {II}},
     PAGES = {857--907},
 PUBLISHER = {Walter de Gruyter, Berlin},
      YEAR = {2004},
      ISBN = {3-11-017478-2},
   MRCLASS = {14F10 (32C38)},
  MRNUMBER = {2099091},
}

@article{Ddaggeraffine,
     author = {Noot-Huyghe, Christine},
     title = {{$D^\dagger (\infty )$}-affinit\'e des sch\'emas projectifs},
     journal = {Annales de l'Institut Fourier},
     pages = {913--956},
     publisher = {Association des Annales de l{\textquoteright}institut Fourier},
     volume = {48},
     number = {4},
     year = {1998},
     doi = {10.5802/aif.1643},
     mrnumber = {2000a:14019},
     zbl = {0910.14005},
     language = {fr},
     url = {https://www.numdam.org/articles/10.5802/aif.1643/}
}

@book{coxtoric,
     author = {Cox, David and Little, John and Schenck, Hal},
     title = {Toric Varieties},
     year={2011},
     series = {Graduate Studies in Mathematics},
     volume = {124},
     publisher = {American Mathematical Society, Providence, RI},
}

@article{Berthelot1,
  author  = {Berthelot, P.},
  title   = {$\mathcal{D}$-modules arithmétiques. {I}. {Opérateurs} différentiels de niveau fini},
  journal = {Annales scientifiques de l'École Normale Supérieure},
  year    = {1996},
  volume  = {29},
  pages   = {185--272}
}

@misc{Caro19,
  author       = {Caro, Daniel},
  title        = {Arithmetic $\mathcal{D}$-modules over algebraic varieties of characteristic $p>0$},
  year         = {2019},
  eprint       = {1909.08432},
  archiveprefix = {arXiv},
  url          = {https://arxiv.org/abs/1909.08432},
  note         = {arXiv:1909.08432},
  primaryclass = {math.AG}
}

@article{toricmodule,
  title={Homological methods for hypergeometric families},
  author={Laura Felicia Matusevich and Ezra Miller and Uli Walther},
  journal={Journal of the American Mathematical Society},
  year={2005},
  volume={18},
  pages={919--941},
  url={https://api.semanticscholar.org/CorpusID:10363377}
}

@article{SW09,
title = {Hypergeometric {D}-modules and twisted {Gau{\ss}--Manin} systems},
journal = {Journal of Algebra},
volume = {322},
number = {9},
pages = {3392--3409},
year = {2009},
note = {Special Issue in Honor of Paul Roberts},
issn = {0021-8693},
doi = {10.1016/j.jalgebra.2008.09.010},
url = {https://www.sciencedirect.com/science/article/pii/S002186930800447X},
author = {Mathias Schulze and Uli Walther}
}

@article {padictorus,
    AUTHOR = {Fu, Lei and Li, Peigen and Wan, Daqing and Zhang, Hao},
     TITLE = {{$p$}-adic {GKZ} hypergeometric complex},
   JOURNAL = {Math. Ann.},
  FJOURNAL = {Mathematische Annalen},
    VOLUME = {387},
      YEAR = {2023},
    NUMBER = {3-4},
     PAGES = {1629--1689},
      ISSN = {0025-5831,1432-1807},
   MRCLASS = {14F30 (11T23 14G15 33C70)},
  MRNUMBER = {4657433},
MRREVIEWER = {Noriko\ Yui},
       DOI = {10.1007/s00208-022-02491-9},
       URL = {https://doi.org/10.1007/s00208-022-02491-9},
}

@article{Carosurhol,
  author  = {Caro, Daniel},
  title   = {$\mathcal{D}$-modules arithmétiques surholonomes},
  journal = {Annales scientifiques de l'École Normale Supérieure},
  series  = {4},
  year    = {2009},
  volume  = {42},
  number  = {1},
  pages   = {141--192},
  doi     = {10.24033/asens.2092},
  url     = {https://www.numdam.org/articles/10.24033/asens.2092/}
}

@article{Carotensorproduct,
  title={Sur la stabilit{\'e} par produit tensoriel de complexes de $\mathcal{D}$-modules arithm{\'e}tiques},
  author={Daniel Caro},
  journal={Manuscripta Mathematica},
  year={2015},
  volume={147},
  pages={1--41},
  url={https://api.semanticscholar.org/CorpusID:122802940}
}

@article {CaroUnipotentmonodromy,
    AUTHOR = {Caro, Daniel},
     TITLE = {Unipotent monodromy and arithmetic {$\mathcal {D}$}-modules},
   JOURNAL = {Manuscripta Math.},
  FJOURNAL = {Manuscripta Mathematica},
    VOLUME = {156},
      YEAR = {2018},
    NUMBER = {1-2},
     PAGES = {81--115},
      ISSN = {0025-2611,1432-1785},
   MRCLASS = {14F10 (14F30)},
  MRNUMBER = {3783567},
MRREVIEWER = {Alberto\ Casta\~no Dom\'inguez},
       DOI = {10.1007/s00229-017-0959-y},
       URL = {https://doi.org/10.1007/s00229-017-0959-y},
}

@article{Ado94,
  title={Hypergeometric functions and rings generated by monomials},
  author={Alan Adolphson},
  journal={Duke Mathematical Journal},
  year={1994},
  volume={73},
  pages={269--290},
  url={https://api.semanticscholar.org/CorpusID:119806807}
}

@article {AS10,
    AUTHOR = {Adolphson, Alan and Sperber, Steven},
     TITLE = {{$A$}-hypergeometric systems that come from geometry},
   JOURNAL = {Proc. Amer. Math. Soc.},
  FJOURNAL = {Proceedings of the American Mathematical Society},
    VOLUME = {140},
      YEAR = {2012},
    NUMBER = {6},
     PAGES = {2033--2042},
      ISSN = {0002-9939,1088-6826},
   MRCLASS = {33C70 (14F40 52B20)},
  MRNUMBER = {2888191},
MRREVIEWER = {Esra\ Erku\c s-Duman},
       DOI = {10.1090/S0002-9939-2011-11073-6},
       URL = {https://doi.org/10.1090/S0002-9939-2011-11073-6},
}

@misc{Fu26,
      title={The {GKZ} hypergeometric $\mathcal{D}$-module},
      author={Lei Fu},
      year={2026},
      eprint={2602.16941},
      archivePrefix={arXiv},
      primaryClass={math.AG},
      url={https://arxiv.org/abs/2602.16941},
      note={arXiv:2602.16941},
}

@article{Fu16,
  title={$\ell$-adic {GKZ} hypergeometric sheaves and exponential sums},
  author={Lei Fu},
  journal={Advances in Mathematics},
  year={2016},
  volume={298},
  pages={51--88},
  url={https://api.semanticscholar.org/CorpusID:123745480}
}

@article {Miyatani20,
    AUTHOR = {Miyatani, Kazuaki},
     TITLE = {{$p$}-adic generalized hypergeometric equations from the
              viewpoint of arithmetic {$\mathcal{D}$}-modules},
   JOURNAL = {Amer. J. Math.},
  FJOURNAL = {American Journal of Mathematics},
    VOLUME = {142},
      YEAR = {2020},
    NUMBER = {4},
     PAGES = {1017--1050},
      ISSN = {0002-9327,1080-6377},
   MRCLASS = {14F10 (14G20 18A99 32C38 33C20 33E50)},
  MRNUMBER = {4124114},
MRREVIEWER = {Alberto\ Casta\~no Dom\'inguez},
       DOI = {10.1353/ajm.2020.0030},
       URL = {https://doi.org/10.1353/ajm.2020.0030},
}

@article {Miyatani22,
    AUTHOR = {Miyatani, Kazuaki},
     TITLE = {Generalized hypergeometric arithmetic {$\mathscr{D}$}-modules under a {$p$}-adic non-{L}iouvilleness condition},
   JOURNAL = {Proc. Amer. Math. Soc.},
  FJOURNAL = {Proceedings of the American Mathematical Society},
    VOLUME = {150},
      YEAR = {2022},
    NUMBER = {8},
     PAGES = {3215--3229},
      ISSN = {0002-9939,1088-6826},
   MRCLASS = {11S80 (14G20)},
  MRNUMBER = {4439447},
MRREVIEWER = {Haoyu\ Hu},
       DOI = {10.1090/proc/15814},
       URL = {https://doi.org/10.1090/proc/15814},
}

@misc{Berrigid,
  title={Cohomologie rigide et cohomologie rigide \`a supports propres},
  author={Berthelot, P.},
  howpublished={Premi\`ere Partie, pr\'epublication, Universit\'e de Rennes},
  year={1996},
  url={https://www.wstein.org/people/berthelo/publis/Cohomologie_Rigide_I.pdf}
}

@article {Caro06,
    AUTHOR = {Caro, D.},
     TITLE = {D\'evissages des {F}-complexes de $\mathcal{D}$-modules arithm\'etiques en {F}-isocristaux surconvergents},
   JOURNAL = {Invent. math.},
    VOLUME = {166},
      YEAR = {2006},
     PAGES = {397--456},
       URL = {https://doi.org/10.1007/s00222-006-0517-9},
}

@article{Carospfunctor,
author = {Caro, Daniel},
year = {2009},
month = {01},
pages = {453--543},
number = {4},
title = {D-modules arithmétiques associés aux isocristaux surconvergents. Cas lisse},
volume = {137},
journal = {Bulletin de la Société Mathématique de France}
}

@article{Caro2011PleineFS,
  title={Pleine fid{\'e}lit{\'e} sans structure de {Frobenius} et isocristaux partiellement surconvergents},
  author={Daniel Caro},
  journal={Mathematische Annalen},
  year={2011},
  volume={349},
  pages={747--805},
  url={https://api.semanticscholar.org/CorpusID:122296409}
}

@article{Tsuzuki99,
     author = {Tsuzuki, Nobuo},
     title = {On the {Gysin} isomorphism of rigid cohomology},
     journal = {Hiroshima Math. J.},
     volume = {29},
     number = {3},
     year = {1999},
     pages = {479--527},
     language = {en},
     url = {http://dml.mathdoc.fr/item/1206124853}
}

@article{BerthelotIntroduction,
author = {Berthelot, P.},
year = {2002},
pages = {1--80},
volume = {279},
title = {{Introduction à la théorie arithmétique des D-modules}},
journal = {Ast\'erisque}
}

@article{Hochster72,
  author = {Hochster, Melvin},
  title = {Rings of invariants of tori, {Cohen--Macaulay} rings generated by monomials, and polytopes},
  journal = {Annals of Mathematics},
  series = {2},
  volume = {96},
  number = {2},
  year = {1972},
  pages = {318--337},
  url = {https://annals.math.princeton.edu/1972/96-2/p05}
}

@article {defnofLiouville,
    AUTHOR = {Christol, Gilles},
     TITLE = {Un th\'eor\`eme de transfert pour les disques singuliers
              r\'eguliers},
   JOURNAL = {Ast\'erisque},
  FJOURNAL = {Ast\'erisque},
    VOLUME = {119--120},
      YEAR = {1984},
     PAGES = {151--168},
      ISSN = {0303-1179,2492-5926},
   MRCLASS = {12H25 (11S80 34A25)},
  MRNUMBER = {773091},
MRREVIEWER = {W.\ Bartenwerfer},
}

@article {CaroTsuzuki12,
    AUTHOR = {Caro, Daniel and Tsuzuki, Nobuo},
     TITLE = {Overholonomicity of overconvergent {$F$}-isocrystals over
              smooth varieties},
   JOURNAL = {Ann. of Math. (2)},
  FJOURNAL = {Annals of Mathematics. Second Series},
    VOLUME = {176},
      YEAR = {2012},
    NUMBER = {2},
     PAGES = {747--813},
      ISSN = {0003-486X,1939-8980},
   MRCLASS = {14F30 (14F10)},
  MRNUMBER = {2950764},
MRREVIEWER = {Adolfo\ Quir\'os},
       DOI = {10.4007/annals.2012.176.2.2},
       URL = {https://doi.org/10.4007/annals.2012.176.2.2},
}

@article {Carologisocrystal,
    AUTHOR = {Caro, Daniel},
     TITLE = {Log-isocristaux surconvergents et holonomie},
   JOURNAL = {Compos. Math.},
  FJOURNAL = {Compositio Mathematica},
    VOLUME = {145},
      YEAR = {2009},
    NUMBER = {6},
     PAGES = {1465--1503},
      ISSN = {0010-437X,1570-5846},
   MRCLASS = {14F30 (14F10)},
  MRNUMBER = {2575091},
MRREVIEWER = {Adolfo\ Quir\'os},
       DOI = {10.1112/S0010437X09004199},
       URL = {https://doi.org/10.1112/S0010437X09004199},
}

@article {Fang,
    AUTHOR = {Fang, Jiangxue},
     TITLE = {Composition series for {GKZ}-systems},
   JOURNAL = {Trans. Amer. Math. Soc.},
  FJOURNAL = {Transactions of the American Mathematical Society},
    VOLUME = {373},
      YEAR = {2020},
    NUMBER = {5},
     PAGES = {3445--3481},
      ISSN = {0002-9947,1088-6850},
   MRCLASS = {14F10 (14M25 32S60 33C80)},
  MRNUMBER = {4082244},
MRREVIEWER = {Bin\ Zhang},
       DOI = {10.1090/tran/7999},
       URL = {https://doi.org/10.1090/tran/7999},
}

@incollection{Kedlaya22,
  author    = {Kedlaya, Kiran S.},
  title     = {{Frobenius} structures on hypergeometric equations},
  booktitle = {Arithmetic, Geometry, Cryptography, and Coding Theory 2021},
  series    = {Contemporary Mathematics},
  volume    = {779},
  pages     = {133--158},
  publisher = {American Mathematical Society},
  address   = {Providence, RI},
  year      = {2022},
  doi       = {10.1090/conm/779/15673}
}

@book{Dwork90,
  author    = {Dwork, Bernard},
  title     = {Generalized Hypergeometric Functions},
  series    = {Oxford Mathematical Monographs},
  publisher = {Clarendon Press, Oxford University Press},
  address   = {Oxford},
  year      = {1990},
  isbn      = {9780198535676},
  doi       = {10.1093/oso/9780198535676.001.0001}
}

@book{Katz90,
  author    = {Katz, Nicholas M.},
  title     = {Exponential Sums and Differential Equations},
  series    = {Annals of Mathematics Studies},
  volume    = {124},
  publisher = {Princeton University Press},
  address   = {Princeton, NJ},
  year      = {1990},
  doi       = {10.1515/9781400882434}
}

@book{FuAG2006,
  author    = {Fu, Lei},
  title     = {Algebraic Geometry},
  publisher = {Tsinghua University Press},
  address   = {Beijing},
  year      = {2006},
  isbn      = {9787302140801}
}

@article{NootHuyghe2007,
  author  = {Noot-Huyghe, Christine},
  title   = {Finitude de la dimension homologique d'alg{\`e}bres
             d'op{\'e}rateurs diff{\'e}rentiels faiblement compl{\`e}tes
             et {\`a} coefficients surconvergents},
  journal = {Journal of Algebra},
  volume  = {307},
  number  = {2},
  year    = {2007},
  pages   = {499--540},
  doi     = {10.1016/j.jalgebra.2004.09.036}
}

\end{document}